\documentclass[a4paper,12pt]{article}
\usepackage{latexsym}
\usepackage{amsmath}
\usepackage{amsthm}

\usepackage{graphicx}	
\usepackage{txfonts}
\usepackage{bm}
\usepackage{color}
\usepackage{epic,eepic}

\usepackage{geometry}
\numberwithin{equation}{section}

\newtheorem{theorem}{Theorem}[section]
\newtheorem{proposition}[theorem]{Proposition}
\newtheorem{remark}{Remark}[section]
\newtheorem{example}{Example}[section]

\newcommand{\OMIT}[1]{{\bf [OMIT:} #1 \ {\bf --- end OMIT] }}  
   \renewcommand{\OMIT}[1]{}            

\newcommand{\embf}{\bf}

\newcommand{\qedJIAM}{\qed}
 \renewcommand{\qedJIAM}{}
\newcommand{\finboxARX}{\finbox}

\newcommand{\RR}{{\bf R}}
\newcommand{\ZZ}{{\bf Z}}

\newcommand{\finbox}{\hspace*{\fill}$\rule{0.17cm}{0.17cm}$}

\newcommand{\odotZ}{\overset{....}} 

\usepackage{lineno}
\newcommand*\patchAmsMathEnvironmentForLineno[1]{
  \expandafter\let\csname old#1\expandafter\endcsname\csname #1\endcsname
  \expandafter\let\csname oldend#1\expandafter\endcsname\csname end#1\endcsname
  \renewenvironment{#1}
     {\linenomath\csname old#1\endcsname}
     {\csname oldend#1\endcsname\endlinenomath}}
\newcommand*\patchBothAmsMathEnvironmentsForLineno[1]{
  \patchAmsMathEnvironmentForLineno{#1}
  \patchAmsMathEnvironmentForLineno{#1*}}
\AtBeginDocument{
\patchBothAmsMathEnvironmentsForLineno{equation}
\patchBothAmsMathEnvironmentsForLineno{align}
\patchBothAmsMathEnvironmentsForLineno{flalign}
\patchBothAmsMathEnvironmentsForLineno{alignat}
\patchBothAmsMathEnvironmentsForLineno{gather}
\patchBothAmsMathEnvironmentsForLineno{multline}
}

\begin{document}

\title{
Decreasing Minimization on Base-Polyhedra: 
Relation Between Discrete and Continuous Cases
}

\author{Andr\'as Frank\thanks{MTA-ELTE Egerv\'ary Research Group,
Department of Operations Research, E\"otv\"os University, P\'azm\'any
P. s. 1/c, Budapest, Hungary, H-1117. 
e-mail: {\tt frank@cs.elte.hu}. 
ORCID: 0000-0001-6161-4848.
The research was partially supported by the
National Research, Development and Innovation Fund of Hungary
(FK\_18) -- No. NKFI-128673.
}
\ \ and \ 
{Kazuo Murota\thanks{
The Institute of Statistical Mathematics,
Tokyo 190-8562, Japan.
Also
Faculty of Economics and Business Administration,
Tokyo Metropolitan University, Tokyo 192-0397, Japan, 
e-mail: {\tt murota@tmu.ac.jp}. 
ORCID: 0000-0003-1518-9152.
The research was supported by 
JSPS KAKENHI Grant Number JP20K11697. 
}}}

\date{February 2022 / April 2022}

\maketitle

\begin{abstract} 
This paper is concerned with the relationship between the discrete
and the continuous decreasing minimization problem on base-polyhedra.
The continuous version (under the name of lexicographically optimal
base of a polymatroid) was solved by Fujishige in 1980,
with subsequent elaborations described in his book (1991).
The discrete counterpart of the dec-min problem 
(concerning M-convex sets) was settled only recently 
by the present authors, with a strongly polynomial algorithm 
to compute not only a single decreasing minimal element 
but also the matroidal structure of all decreasing minimal elements and 
the dual object called the canonical partition.
The objective of this paper is to offer a complete picture
on the relationship 
between the continuous and discrete dec-min problems
on base-polyhedra
by establishing novel technical results and 
integrating known results.
In particular, we derive proximity results, 
asserting the geometric closeness of the
decreasingly minimal elements in the continuous and discrete cases,
by revealing the relation between 
the principal partition 
and the canonical partition.
We also describe decomposition-type algorithms for the discrete case
following the approach of Fujishige and Groenevelt.
\end{abstract}

{\bf Keywords}:  
Base-polyhedron, 
M-convex set, 
Decreasing minimization, 
Lexicographic optimization, 
Principal partition,
Decomposition algorithm.

{\bf Mathematics Subject Classification (2010)}: 90C27, 05C, 68R10




\newpage
\tableofcontents 

\newpage


\section{Introduction}
\label{SCintro}

An element of a set of vectors, in general, 
is called {\embf decreasingly minimal} (dec-min)
if its largest component is as small as possible, within this, 
its second largest component is as small as possible, and so on.
The term ``decreasing minimization'' means the problem of finding
a dec-min element of a given set of vectors
(or a minimum cost dec-min element with respect to a given linear cost-function).  
When the given set of vectors consists of integral vectors, 
this problem is also referred to as discrete decreasing minimization.  
In the literature, typically the term lexicographic 
optimization is used, but we prefer ``decreasing minimization''
because we also consider its natural counterpart 
``increasing maximization,'' and the use of these two symmetric terms
seems more appropriate to distinguish the two related notions.  
An element of a set of vectors is called {\embf increasingly maximal}
(inc-max) if its smallest component is as large as possible, within
this, its second smallest component is as large as possible, and so on.
For example, 
the vector $(2,2,1,1)$ is both dec-min and inc-max in the set
$\{ (2,2,1,1), (2,2,2,0), (2,1,3,0) \}$.
In another set $\{ (2,0,0,0), (1,-1,1,1) \}$,
the dec-min element is $(1,-1,1,1)$
and the inc-max element is $(2,0,0,0)$.

The decreasing minimization (even its weighted form) on a base-polyhedron $B$ 
was investigated by Fujishige \cite{Fuj80} around 1980
under the name of lexicographically optimal bases,
as a generalization of lexicographically optimal maximal flows
considered by Megiddo \cite{Meg74}.
(A lexicographically optimal base in \cite{Fuj80}
means an inc-max member of $B$,
and a lexicographically optimal maximal flow in \cite{Meg74}
means a maximum flow that is inc-max on the set of source-edges.)
Lexicographically optimal bases are discussed in detail in the
book of Fujishige \cite[Sections 8 and 9]{Fuj05book}.
Among others, two important features of base-polyhedra are discovered. 
The first is that the unique dec-min element of a base-polyhedron $B$ 
is the unique inc-max element of $B$, 
while the second is that the unique dec-min element of $B$ 
is the unique square-sum minimizer 
(that is, the minimum $\ell_{2}$-norm element) of $B$.  
These coincidences are surprising in the light that 
neither of the analogous statements hold
for the intersection of two base-polyhedra \cite{FM21partA}.  
Furthermore, the concept
of principal partition of the ground-set defined by
a base-polyhedron
\cite{Fuj09bonn,Iri79} plays a crucial role,
and the critical values associated with the partition
characterizes the lexicographically optimal base (see Section~\ref{SCprinpatR}).
Fujishige developed two algorithms for finding the dec-min element of
base-polyhedron $B$.  
His ``monotone algorithm'' 
\cite[Section~9.2]{Fuj05book}
is not polynomial in its original form but it can immediately be made
strongly polynomial with the aid of the Newton--Dinkelbach algorithm 
\cite{GGJ17,Rad13}.
The other algorithm 
\cite[Section~8.2]{Fuj05book}, 
called ``decomposition algorithm,'' 
is strongly polynomial and does not rely on Newton--Dinkelbach.
In game theory,  
the lexicographically optimal base
was discussed 
under the name of egalitarian allocation
by Dutta and Ray \cite{DR89}
with reference to the framework of majorization \cite{AS18,MOA11};
see also Dutta \cite{Dut90}.

The discrete version of the problem considers dec-min elements of 
the set of integral elements of an integral base-polyhedron,
which set is called an M-convex set in discrete convex analysis
\cite{Mdca98,Mdcasiam,Mbonn09};
the M-convex set arising from an integral base-polyhedron $B$
is denoted as $\odotZ{B}$.
Discrete decreasing minimization on an M-convex set
has been studied recently by the present authors
\cite{FM21partA,FM21partB}.
It was shown in \cite{FM21partA}
that there are interesting coincidences with the continuous case.  
For example, an element $z$ of an M-convex set $\odotZ{B}$ is
dec-min if and only if $z$ is inc-max and if and only if $z$ is a
square-sum minimizer.  
On the other hand, there are fundamental differences between 
the discrete and continuous cases.  
For example, it was shown in \cite{FM21partA}
that there may be several dec-min elements of an M-convex set $\odotZ{B}$, 
and the set of dec-min elements of $\odotZ{B}$ is itself an M-convex set, 
which can be obtained from a matroid by translating the set of characteristic
vectors of its bases by an integral vector.  
This property made it possible to solve algorithmically even 
the minimum cost dec-min problem for M-convex sets (a problem which makes no sense in the
continuous case where the dec-min element is unique).  
Furthermore, as a discrete counterpart of the principal partition, the concept of 
``canonical partition'' of the ground-set was introduced to
characterize the set of all dec-min elements of an M-convex set 
(see Section~\ref{SCcanopatZ}).
The subsequent paper \cite{FM21partB} presented
a strongly polynomial algorithm for computing a dec-min element of an M-convex set
together with the canonical partition,
and discussed applications to a variety of
problems including graph orientations \cite{Fra11book} and
resource allocation problems \cite{HLLT06,IK88,KSI13}.
This algorithm relied on a discrete variant of the Newton--Dinkelbach algorithm, 
and it may be considered a discrete counterpart 
of the ``monotone algorithm'' of Fujishige.  
While the dec-min problem on an M-convex set contains
the discrete version of  Megiddo's problem as a special case,
it does not capture an extension of Megiddo's problem 
in which we seek a feasible integral flow that is inc-max (or dec-min) 
on an arbitrarily specified set of edges.
This more general problem for network flows was investigated recently in \cite{FM22partC},
and a further generalization to submodular integral flows in \cite{FM20partD}.

The objective of this paper is to offer a complete picture
on the relationship 
between the continuous and discrete dec-min problems
on base-polyhedra
by establishing novel technical results and 
integrating known results from the literature. 
As such this paper is partly a research paper and partly a survey paper.

Our novel technical results include
a theorem that reveals a precise relation between the dual objects, namely, 
the principal partition of the ground-set in the continuous case 
and the canonical partition in the discrete case.
This theorem is used to prove proximity results,
asserting the geometric closeness of the
decreasingly minimal elements in the continuous and discrete cases.
The obtained proximity results, in turn, give rise to 
continuous relaxation algorithms
for computing a dec-min element of an M-convex set
in strongly polynomial time
using the (fractional) dec-min element of a base-polyhedron.

Another major topic of this paper
is concerned with decomposition-type algorithms 
for computing a dec-min element of an M-convex set.
The decomposition method was initiated by Fujishige \cite{Fuj80} 
for computing a dec-min element of a base-polyhedron,
or equivalently 
for minimizing a quadratic function on a base-polyhedron,
and was extended by Groenevelt \cite{Gro91}
to separable convex functions on a base-polyhedron
and also on the set of integral points of an integral base-polyhedron
 (namely, an M-convex set).
With a scrutiny of what is known in the literature
about these decomposition algorithms,
we give precise description of two decomposition algorithms
for an M-convex set in a self-contained manner,
one based on Groenevelt \cite{Gro91} and the other based on
Fujishige \cite[Section~8.2]{Fuj05book}.

The paper is organized as follows.
Major ingredients and properties 
of decreasing minimization on base-polyhedra and M-convex sets
are compared in Section~\ref{SCcompRZ}.
In Section~\ref{SCprincanopart}
we reveal the precise relation between 
the principal partition and the canonical partition 
through a novel characterization of the canonical partition.
Using these results,
we derive, in Section~\ref{SCproximity},
proximity results and a continuous relaxation algorithm 
for computing a dec-min element of an M-convex set.
Finally, in Section~\ref{SCdecalg}, 
we deal with decomposition algorithms 
to compute a dec-min element of an M-convex set.

\paragraph{Notation}

We basically follow notation in \cite{FM21partA}.
Let $S$ be a finite ground-set.
For a vector $x\in \RR\sp{S}$ or a function $x:S\rightarrow \RR$, 
we define the set-function 
$\widetilde x:2\sp{S}\rightarrow \RR$ by 
$\widetilde x(Z):=\sum [x(s):s\in Z]$ \ $(Z\subseteq S)$. 
The characteristic (or incidence) vector of a subset $Z \subseteq S$ is denoted by 
$\chi_{Z}$,  that is, 
$\chi_{Z}(s)=1$  if $s\in Z$ and 
$\chi_{Z}(s)=0$ otherwise.
The vector with all components 1 is denoted by $\bm{1}$,
while $\bm{0}$ is the zero vector.
For any real number $\alpha \in \RR$, 
$\lfloor \alpha \rfloor $ denotes the largest integer not larger than $\alpha$, 
and $\lceil \alpha \rceil$ the smallest integer not smaller than $\alpha$. 
This notation is extended to vectors
by componentwise applications.
For any integral polyhedron $P\subseteq \RR \sp{S}$, 
we use the notation $\odotZ{P}$ to denote the set of integral elements of $P$,
that is, 
$\odotZ{P} := P\cap \ZZ\sp{S}$, 
where $\odotZ{P}$ may be pronounced  ``dotted $P$.''  
The notation is intended to refer intuitively to the set of lattice points of $P$.

\section{Comparison of continuous and discrete cases}
\label{SCcompRZ}

In this section we compare the 
decreasing minimization problems on a base-polyhedron
and on an M-convex set 
in terms of various aspects.
Most of them are based on our present knowledge
from \cite{FM21partA,FM21partB,Fuj80,Fuj05book},
while some others serve as motivations 
for the investigations to be made in the present paper.
From these comparisons, it may safely be said 
that the discrete case is,
in spite of some important similarities,
significantly different from the continuous case,
being endowed with a number of intriguing combinatorial structures
on top of the geometric structures known in the continuous case. 
The continuous case is referred to as Case~$\RR$
and the discrete case as Case~$\ZZ$.
We use notation $m_{\RR}$ and $m_{\ZZ}$  
for a dec-min element in Case~$\RR$ and Case~$\ZZ$, respectively.

\subsection{Underlying set}

Let $b$ be a set-function on a ground-set $S$
with $b(\emptyset)=0$,
for which $b(X) =+\infty $ is allowed but
$b(X) = -\infty$ is not.  
We say that $b$ is {\embf submodular} (or fully submodular) if the submodular inequality 
\begin{equation} \label{submodineq}
b(X) + b(Y) \geq b(X\cap Y) + b(X\cup Y)
\end{equation}
holds for every pair of subsets $X, Y\subseteq S$ with finite $b$-values.
A set-function $p$ with $p(\emptyset)=0$,
for which $p(X) = -\infty $ is allowed but
$p(X) = + \infty$ is not, is said to be {\embf supermodular}
if $-p$ is submodular.

For a submodular set-function $b$ on $S$ 
with $b(S)$ finite,
the {\embf base-polyhedron} $B$ 
is defined by 
\begin{equation} \label{basepolysubmod}
B =B(b):=
 \{x\in \RR\sp{S}:  \widetilde x(S)=b(S), \ \widetilde x(Z)\leq b(Z)
  \ \mbox{ for every } \  Z\subset S\},
\end{equation}  
which is possibly unbounded
but never empty. 
The empty set, however, is also considered a base-polyhedron by convention.
When $b$ is integer-valued, $B(b)$ is an integral polyhedron,
which is referred to as {\embf integral base-polyhedron}.
A non-empty base-polyhedron $B$ can also be defined by a supermodular
function $p$ 
with $p(S)$ finite 
as follows:  
\begin{equation} \label{basepolysupermod}
B=B'(p):=
\{x\in \RR\sp{S}:  \widetilde x(S)=p(S), \ \widetilde x(Z)\geq p(Z) \ 
 \mbox{ for every } \ Z\subset S\}.
\end{equation}
We have $B'(p)=B(b)$
if $p$ is a complementary function of $b$,
that is, if  $p(X)=b(S)-b(S-X)$ for all $X \subseteq S$.

The set $\odotZ{B}$ 
of integral elements of an integral base-polyhedron $B$
is called an {\embf M-convex set}
in discrete convex analysis \cite{Mdca98,Mdcasiam,Mbonn09}.
To be more precise, 
an M-convex set
is defined as a set of integral points 
satisfying a certain exchange axiom, and it is
known that 
these two properties are equivalent (\cite[Theorem 4.15]{Mdcasiam}).  
It should be mentioned, however, that the set of integer elements
of an integral base-polyhedron itself 
has long been recognized as a combinatorially nice object
(although no particular name was coined). 
Indeed, this recognition is already evident in Edmonds' classic paper \cite{Edm70},
and Fujishige's book \cite{Fuj05book}
presents the theory of submodular systems
over an arbitrary totally additive group,
of which the set $\ZZ$ of integers is a special case.
It is noted that the set of integral elements of a non-integral base-polyhedron 
is not necessarily an M-convex set.

In Case~$\RR$ of the dec-min problem, we 
seek a dec-min element of a base-polyhedron $B$ 
described by a real-valued supermodular function $p$ or a submodular function $b$.
In Case~$\ZZ$, the dec-min problem is defined on an M-convex set
$\odotZ{B}$, which is the set of integral members of 
an integral base-polyhedron $B$
described by an integer-valued $p$ or $b$.

\subsection{Decreasing minimality and increasing maximality}

In Case~$\RR$ the terminology of 
``lexicographically optimal base''
(or ``lexico-optimal base'')
is used in \cite{Fuj80,Fuj05book}.
A lexico-optimal base is the same as an inc-max element 
in our terminology,
whereas a dec-min element is called a
``co-lexicographically optimal base'' in \cite{Fuj05book}.

In both Case~$\RR$ and Case~$\ZZ$,
decreasing minimality is equivalent to increasing maximality.

\begin{theorem}[\cite{Fuj80,Fuj05book}]  \label{THdecincR}
The unique decreasingly minimal element of $B$ 
is the unique increasingly maximal element of $B$.
\finboxARX
\end{theorem}

\begin{theorem}[{\cite[Theorem~3.3]{FM21partA}}]  \label{THdecincZ}
An element of an M-convex set $\odotZ{B}$ 
is decreasingly minimal in $\odotZ{B}$
if and only if 
it is increasingly maximal in $\odotZ{B}$.  
\finboxARX
\end{theorem}

\subsection{Characterizations}

Let $m$ be an element of an M-convex set $\odotZ{B}$.  
A {\embf 1-tightening step} replaces $m$ by $m':=m+\chi_{s}-\chi_{t}$,
where $s$ and $t$ are elements of $S$
for which $m(t)\geq m(s)+2$ and $m'$ belongs to $\odotZ{B}$.  
A subset $X\subseteq S$ is called {\embf $m$-tight}
(with respect to $p$)
if $\widetilde m(X)=p(X)$.  
A subset $X\subseteq S$ is called an {\embf $m$-top} set
if $m(s)\geq m(t)$ holds whenever $s \in X$ and $t \in S-X$.
(For example, for $m = (3,2,2)$ indexed by 
$S=\{ s_{1},s_{2},s_{3} \}$,
there are five $m$-top sets:
the empty set, 
$\{ s_{1} \}$, 
$\{ s_{1},s_{2} \}$,
$\{ s_{1},s_{3} \}$,
and $S$.)
We call an integral vector $x\in \ZZ\sp{S}$ 
{\embf near-uniform} on a subset $S'$ of $S$ if
its largest and smallest components on $S'$ differ by at most 1, that is, if
there exists some integer $\ell$ for which 
$x(s)\in \{\ell, \ell +1\}$ 
for every $s\in S'$.

The following theorem from \cite{FM21partA} gives 
fundamental characterizations of a dec-min element of an M-convex set.

\begin{theorem}[{\cite[Theorem~3.3]{FM21partA}}]  \label{THdecmincharZ}
For an element $m$ of an M-convex set 
$\odotZ{B} = \odotZ{B'}(p)$,
the following three conditions are pairwise equivalent.
\smallskip

\noindent 
{\rm (A)} \ $m$ is decreasingly minimal in $\odotZ{B}$.

\noindent
{\rm (B)} \ There is no 1-tightening step for $m$.
That is, 
\begin{equation} \label{decminoptexcZ0}
 m(t) \geq m(s) + 2
\ \Longrightarrow \ 
 m + \chi_{s} - \chi_{t} \notin \odotZ{B} .
\end{equation}

\noindent 
{\rm (C)} \ There is a chain 
$(\emptyset \subset) \ C_{1}\subset C_{2}\subset \cdots \subset C_{\ell} \ ( = S)$ 
such that 
each $C_{i}$ is an $m$-top and $m$-tight set (with respect to $p$) 
and $m$ is near-uniform on each 
$S_{i}:= C_{i}-C_{i-1}$ $(i=1,2,\dots ,\ell)$,
where $C_{0}:=\emptyset$.
\finboxARX
\end{theorem}

The corresponding theorem for Case~$\RR$ can be formulated
from known results \cite{Fuj80,Fuj05book}
as follows, where a direct proof can also be obtained from an easy
adaptation of the proof of 
\cite[Theorem~3.3]{FM21partA}.
For an element $m$ of a base-polyhedron $B =B'(p)$
defined by a real-valued supermodular function $p$,  
a subset $X\subseteq S$ is called {$m$-tight}
(with respect to $p$)
if $\widetilde m(X)=p(X)$,
and an {$m$-top} set if $m(s)\geq m(t)$ holds whenever $s \in X$ and $t \in S-X$.
We call a vector $x\in \RR\sp{S}$ 
{\embf uniform} on a subset $S'$ of $S$ if
$x(s) = x(t)$ for all $s, t \in S'$.

\begin{theorem} \label{THdecmincharR}
For an element $m$ of a base-polyhedron $B = B'(p)$,
the following three conditions are pairwise equivalent.
\smallskip

\noindent 
{\rm (A)} \ $m$ is decreasingly minimal in $B$.

\noindent
{\rm (B)} \ 
$m$ satisfies the following condition:
\begin{equation} \label{decminoptexcR0}
 m(t) > m(s), \ \alpha > 0 
\ \Longrightarrow \ 
 m + \alpha( \chi_{s} - \chi_{t} ) \notin B .
\end{equation}

\noindent 
{\rm (C)} \ There is a chain 
$(\emptyset \subset) \ C_{1}\subset C_{2}\subset \cdots \subset C_{\ell} \ ( = S)$ 
such that 
each $C_{i}$ is an $m$-top and $m$-tight set (with respect to $p$) 
and $m$ is uniform on each 
$S_{i}:= C_{i}-C_{i-1}$ $(i=1,2,\dots ,\ell)$,
where $C_{0}:=\emptyset$.
\finboxARX
\end{theorem}

\subsection{Uniqueness}
\label{SCcmparUni}
The structures of dec-min elements
have a striking difference in Case~$\RR$ and Case~$\ZZ$. 
In Case~$\RR$ the dec-min element $m_{\RR}$ of $B$
is uniquely determined.
(The dec-min element, if any, is uniquely determined in an arbitrary polyhedron.) \ 
In Case~$\ZZ$ 
the dec-min elements $m_{\ZZ}$ of $\odotZ{B}$
are endowed with the structure of basis family of a matroid.
This is stated in Theorem~\ref{THmatroideltolt} below,
where a {\embf matroidal M-convex set} means an M-convex set
in which the $\ell_{\infty}$-distance of any two distinct members is equal to one.

\begin{theorem}[{\cite[Theorem~5.7]{FM21partA}}] \label{THmatroideltolt}
The set of dec-min elements of an M-convex set $\odotZ{B}$ is a matroidal M-convex set.
In other words, there exist 
a matroid $M\sp{*}$ and an integral vector $\Delta \sp{*}$
such that
an element $m$ of $\odotZ{B}$ is decreasingly minimal
if and only if $m$ can be obtained as $m=\chi_L+ \Delta \sp{*}$ 
with a basis $L$ of $M\sp{*}$.
\finboxARX
\end{theorem}

The minimum cost dec-min element problem is to compute
a dec-min element that has the smallest cost
with respect to a given cost-function
$c: S \to \RR$ on the ground-set $S$.
In Case~$\ZZ$, this problem is meaningful and interesting,
and  was solved in \cite[Section~5.3]{FM21partA}
on the basis of Theorem~\ref{THmatroideltolt}; see 
\cite{FM21partB} for its instances in graph orientation problems.
In Case~$\RR$, in contrast, this problem
does not make any sense because of the uniqueness
of the dec-min element of $B$.

\subsection{Square-sum minimization}
\label{SCcmparSqsm}
In both Case~$\RR$ and Case~$\ZZ$,
a dec-min element is characterized as a minimizer of
square-sum of the components:
\begin{equation}  \label{Wxdef}
W(x) := \sum [x(s)\sp{2}:  s\in S].
\end{equation}

\begin{theorem}[{\cite[Theorem~3.3]{Fuj80}}]  \label{THdecminRsqsum} 
An element $m$ of a base-polyhedron $B$ is a square-sum minimizer on $B$
if and only if 
$m$ is a dec-min element of $B$.  
\finboxARX
\end{theorem}

\begin{theorem}[{\cite[Corollary~6.4]{FM21partA}}]  \label{THdecminZsqsum} 
An element $m$ of an M-convex set $\odotZ{B}$
is a square-sum minimizer on $\odotZ{B}$
if and only if 
$m$ is a dec-min element of $\odotZ{B}$.  
\finboxARX
\end{theorem}  

In Case~$\RR$, the variable $x$ for minimization is a real vector, $x \in \RR\sp{S}$,
whereas in Case~$\ZZ$ it is an integer vector,  $x \in \ZZ\sp{S}$.
In Case~$\RR$, the minimizer of $W(x)$ over $B$ is unique, 
and is often referred to as the {\embf minimum norm point} of $B$.
(Actually, the minimizer of $W(x)$ is unique for an arbitrary polyhedron,
and is a rational vector for  a rational polyhedron.
However, the minimum norm point may not be the dec-min element.)

In both Case~$\RR$ and Case~$\ZZ$, there are min-max formulas for the square-sum.
The min-max formulas refer to the well-known concept of 
linear extension (or Lov\'asz extension) $\hat p$ of $p$,
which is defined for $\pi \in \RR\sp{S}$ by
\begin{equation}  \label{lovextdef}
\hat p(\pi) := p(S_{n})\pi(s_{n}) + \sum_{j=1}\sp {n-1} p(S_{j})[\pi(s_{j}) - \pi(s_{j+1})],
\end{equation}
where $n= | S |$, the elements of $S$ are indexed in
such a way that 
$\pi(s_{1})\geq \pi(s_{2})\geq  \cdots \geq \pi(s_n)$, and 
$S_{j}:=\{ s_{1}, s_{2},\dots ,s_{j}\}$ for $j=1,2,\dots ,n$.  
Here $p(S_{j})[\pi(s_{j}) - \pi(s_{j+1})]$ is defined to be $0$ when 
$\pi(s_{j}) - \pi(s_{j+1})=0$ even if 
$p(S_{j}) = -\infty$.
In Case~$\ZZ$ we have the min-max identity
\cite[Theorem 6.10]{FM21partA}:

\begin{align} 
& \min \{ \sum [m(s)\sp{2}:  s\in S]:  m\in \odotZ{B} \} 
\nonumber \\ &
= \max \{\hat p(\pi ) - \sum_{s\in S} 
 \left\lfloor \frac{\pi(s)}{2} \right\rfloor 
 \left\lceil  \frac{\pi(s)}{2} \right\rceil : 
 \pi \in \ZZ\sp S \} .
\label{minmaxSqSum-KM3}
\end{align}
In Case~$\RR$, the corresponding formula is
\begin{align} 
& \min \{ \sum [m(s)\sp{2}:  s\in S]:  m\in B \} 
\nonumber \\ &
= \max \{\hat p(\pi) - \sum_{s\in S} 
 \left( \frac{\pi(s)}{2} \right)\sp{2} : \pi \in \RR \sp S \} ,
\label{minmaxsqrsumR}
\end{align}
which may be regarded as an adaptation of the standard quadratic programming duality
to the case where the feasible region is a base-polyhedron.

\subsection{Principal and canonical partitions}

In either of the continuous and discrete problems,
a certain partition of the ground-set $S$
is  known to play an essential role 
as the dual object 
to characterize dec-min elements.
In Case~$\RR$, the partition is called the principal partition,
which 
is used to characterize the (unique) dec-min element $m_{\RR}$ of $B$.
In Case~$\ZZ$, another partition of $S$,
called the canonical partition,
characterizes the set of all dec-min elements $m_{\ZZ}$ of $\odotZ{B}$.
So far these two notions are introduced independently
and nothing is known about their mutual relation.
In Section~\ref{SCprincanopart}, we shall reveal
the precise relation between these partitions
as well as the associated numbers
called critical values and essential values.

\subsection{Proximity}
In general terms, proximity results refer to statements
that the solutions to discrete and continuous versions
of an optimization problem are 
geometrically close to each other.
In Section~\ref{SCproximity} we shall 
obtain proximity results,
showing that dec-min elements $m_{\ZZ}$ of $\odotZ{B}$ 
are located near the dec-min element $m_{\RR}$ of $B$.
The obtained proximity results, in turn, give rise to 
continuous relaxation algorithms
for computing a dec-min element $m_{\ZZ}$ of an M-convex set
in strongly polynomial time 
when the 
(fractional) dec-min element $m_{\RR}$ of a base-polyhedron is given.

\subsection{Algorithm}
\label{SCcmparAlg}

In Case~$\RR$, the decomposition algorithm of Fujishige \cite{Fuj80} in 1980
was already a strongly polynomial algorithm to find the (unique) dec-min element.
This decomposition algorithm 
for computing the dec-min element,
or equivalently 
for minimizing a quadratic function on a base-polyhedron,
was extended by Groenevelt \cite{Gro91}
to separable convex functions on a polymatroid.

Fujishige gave, in his book \cite{Fuj05book}, 
two algorithms for finding the dec-min element for Case~$\RR$. 
The first algorithm  
called ``decomposition algorithm''
\cite[Section~8.2]{Fuj05book}
is not completely the same as, but a variant of, 
the (original) decomposition algorithm of \cite{Fuj80}.
This variant 
is also strongly polynomial.
The other algorithm called ``monotone algorithm'' 
\cite[Section~9.2]{Fuj05book}
is a procedure that computes iteratively 
the members of the principal partition belonging to $B$, 
as well as the critical values (which define the dec-min element $m_{\RR}$ immediately).  
The algorithm is quite simple and natural but it relies
on a subroutine for computing a certain number $\lambda \sp{*}$ which
is, in the present context, equivalent to computing 
$\max \{ p(X)/ |X|:  \emptyset \neq X \subseteq S \}$.  
Though not mentioned explicitly in \cite{Fuj05book}, 
this $\lambda \sp{*}$ can be computed in strongly polynomial time 
with the help of the Newton--Dinkelbach algorithm,
which follows, for example, from a result of Iwata et al.~\cite{IMS97strong}
for a more general problem. 
See Radzik \cite{Rad13} and Goemans et al.~\cite{GGJ17}
for the Newton--Dinkelbach algorithm.

For Case~$\ZZ$,  a strongly polynomial algorithm 
was devised in \cite{FM21partB}.
The algorithm in \cite{FM21partB} relies on a discrete version of the
Newton--Dinkelbach algorithm, and may be viewed 
as a discrete counterpart of Fujishige's monotone algorithm.  
In particular, the algorithm of \cite{FM21partB} computes iteratively 
the canonical chain and partition belonging to $\odotZ{B}$, 
along with the essential value-sequence and a dec-min element itself.  
It is emphasized that the canonical chain and the essential value-sequence 
provide us with 
a structural description (Theorem~\ref{THmatroideltolt}) 
of the set of all dec-min elements of $\odotZ{B}$,
which, in turn, enables us to compute a minimum cost dec-min element of $\odotZ{B}$ 
with respect to a linear cost-function.
It is noted here that, if a single dec-min element of $\odotZ{B}$ is available, 
the canonical chain and the essential value-sequence 
can be computed quite easily \cite[Algorithm 2.3]{FM21partB}.

By the equivalence of dec-minimization and square-sum minimization
(Theorem~\ref{THdecminZsqsum}),
any algorithm for square-sum minimization solves the dec-min problem.
Groenevelt's decomposition algorithm \cite{Gro91}  
for separable convex minimization 
on the set of integral points of an integral polymatroid
can be adapted to minimizing the square-sum on an M-convex set
in strongly polynomial time, which we describe in Section~\ref{SCdecalgG}.
It is natural to expect that
Fujishige's decomposition algorithm \cite[Section~8.2]{Fuj05book} for Case~$\RR$ 
can be adapted to Case~$\ZZ$ through a piecewise-linear extension
of a given separable convex function in integer variables,
as is suggested in \cite[Section~8.3]{Fuj05book}. 
This approach leads indeed to 
another decomposition algorithm for minimizing the square-sum on an M-convex set,
which we describe in Section~\ref{SCdecalgF}.
To realize strong polynomiality we need to devise 
a non-trivial gadget to cope with complications arising from integrality.

Table~\ref{TBhistoryDA} summarizes the development of the decomposition algorithms 
for Case~$\RR$ and Case~$\ZZ$.
In Section~\ref{SCdecalgGen} we remark on the differences of these decomposition algorithms.

\begin{table}
\caption{Decomposition algorithms for minimization on a base-polyhedron}
\label{TBhistoryDA}
\centering 
\begin{tabular}{l|ccc}
   & Case~$\RR$  & & Case~$\ZZ$ 
\\ \hline
 Quadratic &  Fujishige \cite{Fuj80}   & & (Sec.~\ref{SCdecalg} of this paper) 
\\
 (square-sum) &  $\downarrow$
\\
Separable &  Groenevelt \cite{Gro91}  
           & $\rightarrow$ & Groenevelt \cite{Gro91}
\\
\ \ convex  & $\downarrow$ &  & 
\\
& Fujishige \cite[Sec.~8.2]{Fuj05book}    
           & $\rightarrow$ &  Fujishige \cite[Sec.~8.3]{Fuj05book}
\\ \hline
\end{tabular}
\end{table}

\subsection{Weighting}
In Case~$\RR$ a weight vector was introduced 
to define and analyze lexico-optimality in 
\cite{Fuj80,Fuj05book,Vei71}.
In Case~$\ZZ$ the unweighted problem has been investigated in
\cite{FM21partA,FM21partB}
and the weighted case will be treated in a forthcoming paper.
Although we restrict ourselves, in this paper,  to unweighted problems 
in Case~$\RR$ and Case~$\ZZ$,
some facts about the  weighted dec-min problem
are mentioned below.
(The `weighted' dec-min problem here should not
be confused with the `minimum cost' dec-min problem
discussed after Theorem~\ref{THmatroideltolt}.)

Let $w$ be a positive vector on $S$,
which is assumed to be an integral vector in Case~$\ZZ$.
The decreasing minimization problem with weighting $w$
is to find an element $m$ of $B$ (resp., $\odotZ{B}$ in Case~$\ZZ$) for which the vector 
$( w(s) m(s): s \in S)$
is decreasingly minimal in  $B$ (resp., $\odotZ{B}$ in Case~$\ZZ$).
Such an element is called a $w$-dec-min element.
We define $w$-inc-max in an obvious manner.
The (original) decreasing minimization without weighting
corresponds to the case of $w = (1,1,\ldots,1)$.
(It is noted that the weight vector $w$
in \cite{Fuj80,Fuj05book} corresponds
to the (componentwise) reciprocal of the vector $w$ here.) 
The $w$-dec-min problem should not be confused
with the minimum cost dec-min problem 
with respect to a given linear cost-function on $S$.
Even for the weighted dec-min problem,
we can formulate its minimum cost version,
which is to find the minimum cost $w$-dec-min element
with respect to a given cost-function 
$c: S \to \RR$ on $S$.

According to the results of \cite{Fuj80,Fuj05book,Vei71},
we may say, roughly, that
there is not much difference 
between the weighted and unweighted problems in Case~$\RR$.
In Case~$\ZZ$, however, weighting  causes substantial complications.
For example,
$w$-dec-min and $w$-inc-max coincide in Case~$\RR$, but not in Case~$\ZZ$.
In Case~$\RR$, there is a unique $w$-dec-min element in $B$
and is characterized as the unique minimizer of the weighted square-sum 
\begin{equation}  \label{Wcxdef}
W_{w}(x) := \sum [w(s) x(s)\sp{2}:  s\in S].
\end{equation}
In contrast, square-sum minimization does not characterize
$w$-dec-minimality in Case~$\ZZ$.
As a concrete example, consider the line segment $B$ on the plane $\RR\sp{2}$ 
connecting $(2,0)$ and $(0,2)$.
This $B$ is an integral base-polyhedron, and the corresponding M-convex set 
is $\odotZ{B} = \{ (2,0), (1,1), (0,2) \}$.
For the weight vector $w=(1,2)$,
$m_{1}:=(2,0)$ is the (unique) $w$-dec-min element of $\odotZ{B}$
and 
$m_{2}:=(1,1)$ is the (unique) $w$-inc-max element of $\odotZ{B}$.
Since $W_{w}(m_{1}) =4 > 3 = W_{w}(m_{2})$, 
$w$-dec-minimality is not characterized by square-sum minimality.

\section{Principal partition and canonical partition}
\label{SCprincanopart}

A review of the principal partition 
is offered in Section~\ref{SCprinpatR}
with emphasis on its role in decreasing minimization,
while Section~\ref{SCcanopatZ} is a review of the canonical partition  
for discrete decreasing minimization.
Section~\ref{SCcanopatNew} gives a new characterization of the canonical partition,
which is used in Section~\ref{SCcanopatprinpat}
to clarify the relationship between the principal and canonical partitions.
It is mentioned that this section is based on 
our (unpublished) technical report \cite{FM18partII}.

\subsection{Review of the principal partition}
\label{SCprinpatR}

As is pointed out by Fujishige \cite{Fuj80},
the dec-min element in the continuous case
is closely related to the principal partition.
The principal partition is the central concept 
in a structural theory for submodular functions;
see Iri \cite{Iri79} for an early survey and 
Fujishige \cite{Fuj09bonn} for a comprehensive historical and technical account.
In this section we summarize the results 
that are relevant to the analysis of the dec-min element in the continuous case.
Originally \cite{Fuj80}, the results are stated for a real-valued submodular function,
and the description below is a translation for a real-valued supermodular function 
$p: 2\sp{S} \to \RR \cup \{ -\infty \}$
with $p(\emptyset)=0$.

For any real number $\lambda$,
let $\mathcal{L}(\lambda)$ denote the family of all maximizers of
$p(X) - \lambda |X|$.
Then $\mathcal{L}(\lambda)$ is a ring family (lattice), and we denote 
its smallest member by $L(\lambda)$.
That is, $L(\lambda)$ denotes the smallest maximizer of $p(X) - \lambda |X|$.

The following is a  well-known basic fact.
The proof is included for completeness.

\begin{proposition} \label{PRbasicPrinPatR}
\noindent
{\rm (1)}
If $\lambda > \lambda'$, $X \in \mathcal{L}(\lambda)$, and $Y \in \mathcal{L}(\lambda')$, 
then  $X \subseteq Y$.

\noindent
{\rm (2)}
If $\lambda \geq \lambda'$, then $L(\lambda) \subseteq L(\lambda')$.
\end{proposition} 
\begin{proof}
(1) Let $X \in \mathcal{L}(\lambda)$ and $Y \in \mathcal{L}(\lambda')$
for $\lambda > \lambda'$.
We have
\begin{align}
 p(X) +  p(Y) & \leq p(X \cap Y) + p(X \cup Y),
\\ 
 \lambda |X| + \lambda' |Y|  
 & =  \lambda |X \cap Y| + \lambda' |X \cup Y| 
+ (\lambda - \lambda') |X - Y|
\nonumber \\ &
\geq
  \lambda |X \cap Y| + \lambda' |X \cup Y| .
\label{pripatprf}
\end{align}
{}From these inequalities it follows that
\begin{align*}
& ( p(X) - \lambda |X| ) +  ( p(Y) - \lambda' |Y| ) 
\\ & \leq 
 ( p(X \cap Y) - \lambda |X \cap Y| ) +  ( p(X \cup Y) - \lambda' |X \cup Y| ) .
\end{align*}
Here we also have the reverse inequality $\geq$,
since $X$ is a maximizer for $\lambda$ and $Y$ is a maximizer for $\lambda'$. 
Therefore, we have equality in \eqref{pripatprf},
from which follows $(\lambda - \lambda')|X - Y|=0$.
Since $\lambda - \lambda' > 0$, this implies $X \subseteq Y$.

(2) This follows immediately from (1).
\qedJIAM
\end{proof}

There are finitely many numbers $\lambda$ for which
$| \mathcal{L}(\lambda)| \geq 2$.
We denote such numbers as
$\lambda_{1} > \lambda_{2} > \cdots > \lambda_{r}$,
which are called the {\embf critical values}.
Note that the condition $| \mathcal{L}(\lambda)| \geq 2$
for a critical value 
is equivalent to saying that
the largest element of $\mathcal{L}(\lambda)$
is distinct from the smallest element of $\mathcal{L}(\lambda)$.
Since
the largest element of $\mathcal{L}(\lambda)$
is equal to $L(\lambda - \varepsilon)$
for a sufficiently small $\varepsilon > 0$,
we can say that 
$\lambda$ is a critical value if and only if
$L(\lambda) \not= L(\lambda - \varepsilon)$
for any $\varepsilon > 0$.
Thus we obtain:
\begin{align}
& \emptyset= L(\lambda_{1}) \subset L(\lambda_{1}-\varepsilon) 
= \cdots = L(\lambda_{2}) \subset L(\lambda_{2}-\varepsilon) = \cdots
\nonumber \\ & \qquad
\cdots = L(\lambda_{r}) \subset L(\lambda_{r}-\varepsilon) = S
\label{LlambdaChain}
\end{align}
for sufficiently small $\varepsilon > 0$.

The {\embf principal partition}
$\{ \hat S_{1}, \hat S_{2}, \ldots, \hat S_{r} \}$
of the ground-set $S$ is defined by 
\begin{equation}  \label{prinpatRdef}
 \hat S_{i} 
:= L(\lambda_{i+1}) - L(\lambda_{i})
= L(\lambda_{i} - \varepsilon) - L(\lambda_{i})
\qquad (i=1,2,\ldots, r),
\end{equation}
where 
$L(\lambda_{r+1}) := S$ by convention
and $\varepsilon$ is a sufficiently small positive number.
By defining 
\begin{equation}  \label{prinChainRdef}
\hat C_{i} := 
\hat S_{1} \cup \hat S_{2} \cup \cdots \cup  \hat S_{i}
\quad (= L(\lambda_{i} - \varepsilon) ) 
\end{equation}
for $i=1,2,\ldots, r$,
we obtain a chain: 
\begin{equation}  \label{princhainR}
\hat C_{1} \subset  \hat C_{2} \subset  \cdots \subset \hat C_{r},
\end{equation}
where $\hat C_{1} \not= \emptyset$ and $\hat C_{r} = S$.
In this paper we call this chain the {\embf principal chain}.
We have 
$\hat S_{i}= \hat C_{i} -  \hat C_{i-1}$ for $i=1,2,\ldots, r$,
where $\hat C_{0} := \emptyset$.

\begin{remark} \rm  \label{RMprinfiner}
The principal partition 
$\{ \hat S_{1}, \hat S_{2}, \ldots, \hat S_{r} \}$
defined here by \eqref{prinpatRdef}
is, in fact, an aggregation of 
what is usually meant under the name of ``principal partition'' \cite{Fuj09bonn,Iri79}.
The latter is defined as follows.
Not only each $\mathcal{L}(\lambda)$ is a lattice, but
their union 
$\mathcal{L}_{\rm all} := \bigcup_{\lambda \in \RR} \mathcal{L}(\lambda)$
is also a lattice by Proposition~\ref{PRbasicPrinPatR}(1),
and we have $\mathcal{L}_{\rm all} = \bigcup_{i=1}\sp{r} \mathcal{L}(\lambda_{i})$.
A maximal chain of this lattice $\mathcal{L}_{\rm all}$
induces a partition of the ground-set $S$,
and the induced partition is determined independently of the choice of a maximal chain.
In \cite{Fuj09bonn,Iri79},  this is called the principal partition of $S$.
Furthermore, a partial order can be defined among the members
of the partition.
The chain
\begin{equation}  \label{princhainR0}
(\emptyset =) \ \hat C_{0} \subset  \hat C_{1} \subset  \hat C_{2} 
\subset  \cdots \subset \hat C_{r} \ (=S)
\end{equation}
associated with our partition 
$\{ \hat S_{1}, \hat S_{2}, \ldots, \hat S_{r} \}$
is certainly a chain of $\mathcal{L}_{\rm all}$,
which, however, may not be maximal.
If the chain in \eqref{princhainR0} is not maximal,  
our partition is an aggregation of the principal partition 
in the sense of \cite{Fuj09bonn,Iri79}.
\finbox
\end{remark}

The following theorem shows the close relationship
between the principal partition and the unique dec-min element (minimum norm point)
of a base-polyhedron $B=B'(p)$.

\begin{theorem}[Fujishige \cite{Fuj80}] \label{THlexoptbaseR}
Let $B=B'(p)$ be a base-polyhedron defined by a supermodular function $p$.

\noindent
{\rm (1)} \ 
Let 
$\{ \hat S_{1}, \hat S_{2},  \allowbreak  \ldots, \allowbreak   \hat S_{r} \}$
be the principal partition
and
$\lambda_{1} > \lambda_{2} > \cdots > \lambda_{r}$
the critical values.
The unique dec-min element $m_{\RR}$ of $B$ 
is given by 
$m_{\RR}(s) = \lambda_{i}$ for $s \in  \hat S_{i}$ and $i=1,2,\ldots, r$. 
In particular,  $m_{\RR}$ is uniform on each member $\hat S_{i}$
of the principal partition.

\noindent
{\rm (2)} \ 
Let $m_{\RR}$ be the unique dec-min element of $B$.
The critical values 
$\lambda_{1} > \lambda_{2} > \cdots > \lambda_{r}$
are precisely those numbers that appear as component values of $m_{\RR}$,
and the principal partition
$\{ \hat S_{1}, \hat S_{2}, \ldots, \hat S_{r} \}$
is given by 
$\hat S_{i} = \{ s \in S : m_{\RR}(s) = \lambda_{i} \}$\, 
for $i=1,2,\ldots, r$. 
\finboxARX
\end{theorem}

The following characterization of dec-minimality 
can be formulated by combining
Theorem~\ref{THlexoptbaseR} and Theorem~\ref{THdecmincharR}.

\begin{theorem} \label{THdecmincharR2}
Let $B=B'(p)$ be a base-polyhedron defined by a supermodular function $p$.
An element $m \in B$ with distinct component values 
$\lambda'_{1}> \lambda'_{2}> \cdots >\lambda'_{\ell}$ is 
the unique dec-min element of $B$
if and only if
each ``level set'' 
$\hat C_{i}:=\{ s\in S:  m(s)\geq \lambda'_{i} \}$ is $m$-tight
($\widetilde m(\hat C_{i})=p(\hat C_{i})$)
for $i=1,2,\ldots, \ell$.
\end{theorem}
\begin{proof}
The only-if part is immediate from Theorem~\ref{THlexoptbaseR}(2).
For the if-part,
note that $\hat C_{i}$ is an $m$-top set by definition,
which is also $m$-tight by assumption.
Then we can use 
(C) $\Rightarrow$ (A) in 
Theorem~\ref{THdecmincharR}.
\qedJIAM
\end{proof}

\subsection{Review of the canonical partition}
\label{SCcanopatZ}

In the discrete case,
the canonical partition describes the structure of dec-min elements.
The canonical partition is defined iteratively using contractions
as follows \cite{FM21partA}.

Let $p: 2\sp{S} \to \ZZ \cup \{ -\infty \}$
be an integer-valued supermodular function
with $p(\emptyset)=0$ and $p(S) > -\infty$, and define $C_{0} := \emptyset$.
For $j=1,2,\ldots, q$,	define
\begin{align}
\beta_{j} &:= \max \left\{  \left\lceil  \frac{p(X \cup C_{j-1}) - p(C_{j-1})}{|X|}  \right\rceil :  
         \emptyset \not= X \subseteq \overline{C_{j-1}}   \right\},
\label{betajdef}
\\
h_{j}(X) &:= p(X \cup C_{j-1}) - (\beta_{j} - 1) |X| - p(C_{j-1})    
\qquad
(X \subseteq \overline{C_{j-1}}),
\label{hjdef}
\\
S_{j} &:=  \mbox{smallest subset of $\overline{C_{j-1}}$ maximizing $h_{j}$},
\label{Sjdef}
\\
C_{j} &:=  C_{j-1} \cup S_{j} ,
\label{Cjdef}
\end{align}
where $\overline{C_{j-1}} = S  - C_{j-1}$ and
 the index $q$ is determined by the condition that
$C_{q-1} \not= S$ and $C_{q} = S$.

The resulting chain  
$C_{1} \subset  C_{2} \subset  \cdots \subset  C_{q}$
is called the {\embf canonical chain},
and the partition  
$\{ S_{1},  S_{2}, \ldots, S_{q} \}$
is the {\embf canonical partition}. 
The integers $\beta_{j}$, 
known \cite[Theorem~5.5]{FM21partA} to be decreasing,
are called the {\embf essential values},
and this decreasing sequence 
$\beta_{1} > \beta_{2} > \cdots > \beta_{q}$
is named the {\embf essential value-sequence}.
We have
$S_{j}=C_{j} -  C_{j-1}$ for $j=1,2,\ldots, q$.

The following theorem characterizes dec-min elements of $\odotZ{B}$
using these concepts.

\begin{theorem}[{\cite[Corollary 5.2]{FM21partA}}]  \label{THchardecminZ}
Let $B=B'(p)$ be an integral base-polyhedron on ground-set $S$.  
Let $\{ C_{1}, C_{2}, \dots ,C_{q}\}$ be the canonical chain,
$\{ S_{1},  \allowbreak  S_{2}, \allowbreak   \dots ,S_{q}\}$ 
the canonical partition of $S$,
and $\beta_{1}> \beta_{2}> \dots > \beta_{q}$ 
the essential value-sequence belonging to $\odotZ{B}$.  
Then an element $m \in \odotZ{B}$ is decreasingly minimal in $\odotZ{B}$
if and only if 
each $C_{j}$ is $m$-tight 
(that is, $\widetilde m(C_{j})=p(C_{j})$) 
and $\beta_{j} -1 \leq m(s) \leq \beta_{j}$ holds for each 
$s\in S_{j}$ \ $(j=1,2,\dots ,q)$.  
\finboxARX 
\end{theorem}

This theorem implies, in particular, that 
every dec-min element $m_{\ZZ}$ of $\odotZ{B}$ is near-uniform 
on each member of the canonical partition.
That is, $|m_{\ZZ}(s) - m_{\ZZ}(t)| \leq 1$ if $\{ s, t \} \subseteq S_{j}$ 
for some member $S_{j}$ of the canonical partition.

The following theorem shows that any dec-min element $m_{\ZZ}$ of $\odotZ{B}$,
in turn,
determines the essential value-sequence, the canonical chain, and
canonical partition.

\begin{theorem}[{\cite[Corollary 5.4]{FM21partA}}]  \label{THcanodecminZ}
Let $m_{\ZZ}$ be an arbitrary dec-min element of $\odotZ{B}$.  
The first  essential value $\beta_{1}$ 
is the largest $m_{\ZZ}$-value and 
the first member $C_{1}$ of the canonical chain
is the smallest
$m_{\ZZ}$-tight set containing all $\beta_{1}$-valued elements.  
Moreover, for $j=2,\dots ,q$, 
the $j$-th essential value $\beta_{j}$ is the largest value of 
$m_{\ZZ}(s)$ for $s \in S - C_{j-1}$ and 
the $j$-th member $C_{j}$ of the canonical chain
is the smallest $m_{\ZZ}$-tight set 
(with respect to $p$) 
containing each element of $m_{\ZZ}$-value at least $\beta_{j}$.
\finboxARX
\end{theorem}

These results
(Theorems \ref{THchardecminZ} and \ref{THcanodecminZ})
may be viewed as a discrete
counterpart of Theorem~\ref{THlexoptbaseR} for the continuous case.

\subsection{An alternative characterization of the canonical partition}
\label{SCcanopatNew}

Originally \cite{FM21partA},
the canonical partition was defined iteratively using contractions,
as described by \eqref{betajdef}--\eqref{Cjdef}
in Section~\ref{SCcanopatZ}.
In this section we give a non-iterative construction of this canonical partition,
which reflects the underlying structure more directly.
This alternative construction enables us to reveal,
in Section~\ref{SCcanopatprinpat},
the precise relation 
between the canonical partition and the principal partition.

By the definition given in \eqref{betajdef}--\eqref{Cjdef}, we have that
\begin{equation} \label{CjsmallestmaxCj1}
\mbox{$C_{j}$ is the smallest maximizer of 
$p(X) - (\beta_{j}-1) |X|$
among all $X \supseteq C_{j-1}$}.
\end{equation}
We will show, in Proposition~\ref{PRbasicitera} below,
that  $C_{j}$ is, in fact,  the smallest maximizer of 
$p(X) - (\beta_{j}-1) |X|$
among all subsets $X$ of $S$
without the constraint of being a superset of $C_{j-1}$.

For any integer $\beta$,
let $\mathcal{L}(\beta)$ denote the family of all maximizers of $p(X) - \beta |X|$,
and  $L(\beta)$ be the smallest element of $\mathcal{L}(\beta)$,
where  the smallest element exists in $\mathcal{L}(\beta)$ since 
$\mathcal{L}(\beta)$ is a lattice (ring family).
(These notations are consistent with the ones introduced in Section~\ref{SCprinpatR}.) \

\begin{proposition} \label{PRbasicitera}
Let $\beta_{j}$ and $C_{j}$ $(j=1,2,\ldots,q)$
be defined by \eqref{betajdef}--\eqref{Cjdef}.

\noindent
{\rm (1)}
 $\beta_{1} > \beta_{2} > \cdots > \beta_{q}$.

\noindent
{\rm (2)}
For each $j$ with $1 \leq j \leq q$, $C_{j}$ is 
the smallest maximizer of $p(X) - (\beta_{j}-1) |X|$
among all subsets $X$ of $S$.
\end{proposition} 
\begin{proof}
(1)
The monotonicity of the $\beta$-values is already shown in 
\cite[Theorem~5.5]{FM21partA}, but we give an alternative proof here.
Let $j \geq 2$.
By \eqref{betajdef}, we have $\beta_{j-1} > \beta_{j}$ 
if and only if
\begin{equation} \label{prf1betaj1betaj}
\beta_{j-1} > \left\lceil \frac{p(X \cup C_{j-1}) - p(C_{j-1})}{|X|} \right\rceil 
\end{equation}
for every $X$ with $\emptyset \not= X \subseteq \overline{C_{j-1}}$.
Furthermore, we can rewrite the condition \eqref{prf1betaj1betaj} as follows:
\begin{align*} 
\eqref{prf1betaj1betaj}
& \Leftrightarrow
\beta_{j-1} -1  \geq   \frac{p(X \cup C_{j-1}) - p(C_{j-1})}{|X|}  
\\ & \Leftrightarrow
p(X \cup C_{j-1}) - p(C_{j-1}) \leq (\beta_{j-1} -1 ) |X|
\\ & \Leftrightarrow
p(X \cup C_{j-1}) - (\beta_{j-1} -1 ) |X  \cup C_{j-1}| \leq  p(C_{j-1}) - (\beta_{j-1} -1 ) |C_{j-1}|.
\end{align*}
The last inequality holds for every $X \subseteq \overline{C_{j-1}}$,
since
$C_{j-1}$ is the (smallest) maximizer of 
$p(X) - (\beta_{j-1}-1) |X|$
among all $X$ containing $C_{j-2}$,
and the set $X \cup C_{j-1}$ does contain $C_{j-2}$.
We have thus shown $\beta_{j-1} > \beta_{j}$.

(2)
We prove $C_{j}= L(\beta_{j} -1)$ for $j=1,2,\ldots,q$ by induction on $j$. 
This holds for $j=1$ by definition.
Let $j \geq 2$.
By Proposition~\ref{PRbasicPrinPatR}(2)
for $\lambda = \beta_{j-1}-1$ and $\lambda' = \beta_{j}-1$,
we have 
$L(\beta_{j} -1) \supseteq L(\beta_{j-1} -1)$,
whereas
$L(\beta_{j-1} -1) =C_{j-1}$ by the induction hypothesis.
Combining this with \eqref{CjsmallestmaxCj1}, we obtain $C_{j}= L(\beta_{j} -1)$.
\qedJIAM
\end{proof}

We now give an alternative characterization of the essential value-sequence 
$\beta_{1} > \beta_{2} > \cdots > \beta_{q}$
defined by \eqref{betajdef}--\eqref{Cjdef}.
We consider the family $\{ L(\beta) : \beta \in \ZZ \}$
of the smallest maximizers of $p(X) - \beta |X|$ for all integers $\beta$.
Each $C_{j}$ is a member of this family, since
$C_{j} = L(\beta_{j} - 1)$ ($j=1,2,\ldots, q$)
by Proposition~\ref{PRbasicitera}(2).

\begin{proposition} \label{PRsequenceLbeta}
As $\beta$ is decreased from $+\infty$ to $-\infty$
(or from $\beta_{1}$ to $\beta_{q}-1$), 
the smallest maximizer $L(\beta)$
is monotone non-decreasing.
We have $L(\beta) \not= L(\beta -1)$ if and only if
$\beta$ is equal to an essential value.
Therefore, the essential value-sequence 
$\beta_{1} > \beta_{2} > \cdots > \beta_{q}$
is characterized by the property:
\begin{align}
& \emptyset= L(\beta_{1}) \subset L(\beta_{1}-1) 
= \cdots = L(\beta_{2}) \subset L(\beta_{2}-1) = \cdots
\nonumber \\ & \qquad
\cdots = L(\beta_{q}) \subset L(\beta_{q}-1) = S.
\label{LbetaChain}
\end{align}
\end{proposition} 
\begin{proof}
The monotonicity of $L(\beta)$  follows from Proposition~\ref{PRbasicPrinPatR}(2).
We will show 
(i) $L(\beta_{1}) = \emptyset$,
(ii) $L(\beta_{j-1} -1) = L(\beta_{j})$
for $j=2,\ldots, q$, and
(iii) $L(\beta_{j}) \subset L(\beta_{j}-1)$ for $j=1,2,\ldots, q$.

(i)
Since 
$\beta_{1} = \max \left\{  \left\lceil  p(X)/|X|  \right\rceil  :  X  \not = \emptyset  \right\}$,
we have
$p(X) - \beta_{1} |X| \leq  0$ for all $X \not= \emptyset$,
whereas 
$p(X) - \beta_{1} |X| =  0$ for $X = \emptyset$.
Therefore, $L(\beta_{1}) = \emptyset$.

(ii)
Let $2 \leq j \leq q$.
For short we put $C := C_{j-1}$ and define 
\[
h(Y) := p(Y) - \beta_{j} |Y|
\]
for any subset $Y$ of $S$.
Let $A$ be the smallest  maximizer of $h$, which means 
$A = L(\beta_{j})$.
For any non-empty subset $X$ of $\overline{C} \ (= S - C)$
we have
\begin{align*}
& \beta_{j}
\geq    \left\lceil  \frac{p(X \cup C) - p(C)}{|X|}  \right\rceil 
\geq   \frac{p(X \cup C) - p(C)}{|X|} ,
\end{align*}
which implies
$p(X \cup C) - \beta_{j} |X  \cup C| \leq  p(C) - \beta_{j} |C|$,
that is,
\begin{align}
h(Y) \leq  h(C)
\qquad \mbox{for all \  $Y \supseteq C$}.
\label{prfhYhC}
\end{align}
By supermodularity of $p$ we have
\[
 h(A)  + h(C) \leq  h(A \cup C)  + h(A \cap C),
\]
whereas $h(C) \geq  h(A \cup C)$
by \eqref{prfhYhC}.
Therefore,
$ h(A)  \leq  h(A \cap C)$.
Since $A$ is the smallest maximizer of $h$,
this implies that $A = A \cap C$, i.e., $A \subseteq C$.
On recalling notations
$A = L(\beta_{j})$ and $C = C_{j-1}= L(\beta_{j-1} -1)$,
we obtain 
$L(\beta_{j}) \subseteq L(\beta_{j-1} -1)$.
We also have
$L(\beta_{j}) \supseteq L(\beta_{j-1} -1)$
by the monotonicity.
Therefore,
$L(\beta_{j}) = L(\beta_{j-1} -1)$.

(iii)
Let $1 \leq j \leq q$ and put $C := C_{j-1}$.
Take a non-empty subset $Z$ of $\overline{C}$ which
gives the maximum in the definition of $\beta_{j}$, i.e., 
\begin{align*}
& \beta_{j}
= \max \left\{  \left\lceil  \frac{p(X \cup C) - p(C)}{|X|}  \right\rceil    
 :  \emptyset \not= X \subseteq \overline{C}   \right\} 
 = \left\lceil  \frac{p(Z \cup C) - p(C)}{|Z|}  \right\rceil .
\end{align*}
Then we have
\begin{align*}
\frac{p(Z \cup C) - p(C)}{|Z|}  > \beta_{j} -1 ,
\end{align*}
which implies
\begin{align*}
p(Z \cup C) - (\beta_{j}-1) |Z \cup C| >  p(C) -  (\beta_{j}-1) |C|.
\end{align*}
This shows that 
$C = C_{j-1}= L(\beta_{j-1} -1)$
is not a maximizer of $p(X) - (\beta_{j}-1) |X|$,
and hence  $L(\beta_{j-1} -1) \not= L(\beta_{j}-1)$.
On the other hand, we have 
$L(\beta_{j-1} -1) = L(\beta_{j})$ by (ii) and 
$L(\beta_{j}) \subseteq L(\beta_{j}-1)$
by the monotonicity.
Therefore,
$L(\beta_{j}) \subset L(\beta_{j}-1)$.
\qedJIAM
\end{proof}

Proposition~\ref{PRsequenceLbeta} justifies the following alternative definitions
of the essential value-sequence, the canonical chain, and the canonical partition.

\begin{proposition} \label{PRaltcano}
The essential value-sequence, the canonical chain, and the canonical partition
can also be defined as follows:
\begin{quote}
Consider the smallest maximizer $L(\beta)$ of $p(X) - \beta |X|$ for each integer $\beta$.
There are finitely many $\beta$ for which
$L(\beta) \not= L(\beta -1)$.
Denote such integers as
$\beta_{1} > \beta_{2} > \cdots > \beta_{q}$
and call them the {essential value-sequence}.
Furthermore, define 
$C_{j} := L(\beta_{j}-1)$
for $j=1,2,\ldots, q$
to obtain a chain: 
$C_{1} \subset  C_{2} \subset  \cdots \subset  C_{q}$.
Call this  the {canonical chain}.
Finally define a partition  $\{ S_{1},  S_{2}, \ldots, S_{q} \}$
of $S$ by 
\begin{equation}  \label{canopatZdef2}
S_{j} := C_{j} -  C_{j-1} = L(\beta_{j} - 1) - L(\beta_{j})
\qquad (j=1,2,\ldots, q),
\end{equation}
where $C_{0} := \emptyset$,
and call this the {canonical partition}. 
\end{quote}
\vspace{-\baselineskip} 
\finboxARX
\end{proposition}

This alternative construction clearly exhibits the parallelism between 
the canonical partition in Case~$\ZZ$ and the principal partition in Case~$\RR$.
In particular, the essential values are 
the discrete counterpart of the critical values.
This is discussed in the next section.

\subsection{Canonical partition from the principal partition}
\label{SCcanopatprinpat}

The characterization of the canonical partition shown in Section~\ref{SCcanopatNew}
enables us to construct 
the canonical partition and essential values for Case~$\ZZ$ 
from the principal partition and critical values for Case~$\RR$ as follows.

\begin{theorem} \label{THrelRZpartition}
Let $B$ be an integral base-polyhedron defined by an integer-valued supermodular
function $p$.

\noindent
{\rm (1)}
The essential values $\beta_{1} > \beta_{2} > \cdots > \beta_{q}$
are obtained from the critical values 
$\lambda_{1} > \lambda_{2} >  \cdots >  \lambda_{r}$
as the distinct members of the rounded-up integers
$\lceil \lambda_{1} \rceil \geq \lceil \lambda_{2} \rceil 
  \geq \cdots \geq \lceil \lambda_{r} \rceil$.
In particular,
an integer $\beta$ is an essential value if and only if there exists 
a critical value $\lambda$ satisfying $\beta  \geq \lambda > \beta -1$.

\noindent
{\rm (2)}
The canonical partition 
$\{ S_{1}, S_{2}, \ldots, S_{q} \}$
is an aggregation of the principal partition 
$\{ \hat S_{1}, \hat S_{2},   \allowbreak   \ldots, \hat S_{r} \}$
given by 
\begin{equation} \label{canpataggreprinpat}
 S_{j} = \bigcup_{i \in I(j)} \hat S_{i} 
\qquad (j=1,2,\ldots, q),
\end{equation}
where $I(j) := \{ i :  \lceil \lambda_{i} \rceil  = \beta_{j} \}$
for $j=1,2,\ldots, q$.

\noindent
{\rm (3)}
The canonical chain
$\{ C_{j} \}$
is a subchain of the principal chain  $\{ \hat C_{i} \}$;
we have $C_{j} = \hat C_{i}$ for $i = \max I(j)$.
\end{theorem} 
\begin{proof}
(1)
By Proposition~\ref{PRaltcano}, an integer $\beta$ 
is an essential value if and only if 
$L(\beta) \not= L(\beta -1)$,
whereas a real number
$\lambda$ is a critical value if and only if
$L(\lambda) \not= L(\lambda - \varepsilon)$
for any $\varepsilon > 0$.
Hence follows the claim.

(2)
For each  $j=1,2,\ldots, q$, 
let $i_{\max}$ and $i_{\min}$  denote the maximum and minimum elements of $I(j)$,
respectively.
By \eqref{prinpatRdef}
and \eqref{canopatZdef2} we have
\begin{align*}
\bigcup_{i \in I(j)} \hat S_{i} 
& =\bigcup_{i \in I(j)} ( L(\lambda_{i+1}) - L(\lambda_{i}) )
=  L(\lambda_{i_{\max} + 1}) - L(\lambda_{i_{\min}}) 
\\
& =  L(\lambda_{i_{\max}} - \varepsilon) - L(\lambda_{i_{\min}}) 
=  L(\beta_{j} - 1) - L(\beta_{j}) 
= S_{j}.
\end{align*}

(3)
By \eqref{canpataggreprinpat}
we have
\[
C_{j} = \bigcup_{k \leq j} S_{k}
= \bigcup_{k \leq j} \bigcup_{i \in I(k)} \hat S_{i} 
= \bigcup_{i \leq i_{\max} } \hat S_{i} 
= \hat C_{i_{\max}}.
\]
This completes the proof of Theorem~\ref{THrelRZpartition}.
\qedJIAM
\end{proof}

The following two examples illustrate Theorem~\ref{THrelRZpartition}.
The first example treats the simplest case 
to demonstrate the basic idea as well as the notation.
The second is a running example, 
to be considered repeatedly to illustrate our subsequent arguments.

\begin{example} \rm \label{EX2dimB1ppRZ}
Let $S = \{ s_{1} , s_{2} \}$ and $\odotZ{B} = \{ (0,3), (1,2), (2,1) \}$,
where $B$ is the line segment connecting $(0,3)$ and $(2,1)$.
There are two dec-min elements of $\odotZ{B}$, 
$m_{\ZZ}\sp{(1)}=(1,2)$ and $m_{\ZZ}\sp{(2)}=(2,1)$.
The minimum norm point (dec-min element) of $B$ is
$m_{\RR}=(3/2,3/2)$.
The supermodular function $p$ for $B=B'(p)$ is given by
\[
 p(\emptyset)= 0,  
\quad
 p(\{ s_{1} \})= 0,  
\quad
 p(\{ s_{2} \})= 1,  
\quad
 p(\{ s_{1},s_{2} \}) = 3,
\]
and we have
\begin{equation*} 
 p(X) - \lambda |X| = 
   \left\{  \begin{array}{ll}
    0   & (X = \emptyset),  \\
    -\lambda  & (X = \{ s_{1} \}),  \\
    1-\lambda  & (X = \{ s_{2} \}),  \\
    3- 2\lambda  & (X = \{ s_{1}, s_{2} \}).  \\
             \end{array}  \right.
\end{equation*}
There is only one ($r=1$) critical value $\lambda_{1} = 3/2$ and
the associated sublattice is $\mathcal{L}(\lambda_{1}) = \{ \emptyset, S \}$.
The principal partition is a trivial partition
$\{ \hat S_{1} \}$ 
consisting of only one member $\hat S_{1} = S$.
Since $\lceil \lambda_{1} \rceil = 2$, we have
$\beta_{1} = 2$ with $q=1$, and the (only) member $S_{1}$
in the canonical partition is given by 
$S_{1} = C_{1} = L(\beta_{1}-1)= L(1) = S$.
\finbox
\end{example}

\begin{example} \rm \label{EX4dimC4pat}
Let $S = \{ s_{1}, s_{2}, s_{3}, s_{4}  \}$ and consider five vectors
\begin{align*}
& m_{1}=(2,1,1,0), \ 
 m_{2}=(2,1,0,1), \
 m_{3}=(1,2,1,0), \
\\ &
 m_{4}=(1,2,0,1), \
 m_{5}=(2,2,0,0).
\end{align*}
These vectors are obtained by adding vector $(1,1,0,0)$ to
$(1,0,1,0)$,  $(1,0,0,1)$,   $(0,1,1,0)$,  $(0,1,0,1)$,  $(1,1,0,0)$, 
which are the characteristic vectors of bases of a rank-2 matroid on $S$.
Let $B$ denote the convex hull of $\{ m_{1}, m_{2}, \ldots, m_{5} \}$.
Then $B$ is an integral base-polyhedron and
$\odotZ{B} = \{ m_{1}, m_{2}, \ldots, m_{5} \}$ is an M-convex set.
The dec-min elements of $\odotZ{B}$  are 
$m_{1}$, $m_{2}$, $m_{3}$, and $m_{4}$,
whereas 
$ m_{5}=(2,2,0,0)$ is not dec-min.
The supermodular function $p$ for $B=B'(p)$ is given by
\begin{align*}
& p(\emptyset)= 0, \quad
p(\{ s_{1} \})= p(\{ s_{2} \})= 1, \quad  
p(\{ s_{3} \})= p(\{ s_{4} \})= 0,  
\\ & 
p(\{ s_{1},s_{2} \}) = 3, \quad  
p(\{ s_{3},s_{4} \}) = 0, \quad
\\ & 
p(\{ s_{1},s_{3} \}) =p(\{ s_{2},s_{3} \}) =p(\{ s_{1},s_{4} \}) =p(\{ s_{2},s_{4} \}) = 1,
\\ & 
p(\{ s_{1},s_{2}, s_{3} \}) =
p(\{ s_{1},s_{2}, s_{4} \}) = 3, \quad
\\ & 
p(\{ s_{1},s_{3}, s_{4} \}) =
p(\{ s_{2},s_{3}, s_{4} \}) = 2, \quad
\\ & 
p(\{ s_{1},s_{2},s_{3}, s_{4} \}) = 4.
\end{align*}
We have
\begin{equation*} \label{}  
 \max\{ p(X) - \lambda |X| : X \subseteq S \} = 
 \max\{ 0, \   1- \lambda, \  3- 2\lambda,  \   3- 3\lambda,  \   4- 4\lambda \}.  
\end{equation*}
There are two ($r=2$) critical values $\lambda_{1} = 3/2$ and $\lambda_{2} = 1/2$,
with the associated sublattices
$\mathcal{L}(\lambda_{1}) = \{ \emptyset, \{ s_{1},s_{2} \} \}$
and
$\mathcal{L}(\lambda_{2}) = \{\{ s_{1},s_{2} \}, S \}$.
The principal chain is given by
$\{ s_{1},s_{2} \} \subset \{ s_{1}, s_{2}, s_{3}, s_{4}  \}$,
and the principal partition is a bipartition with 
$\hat S_{1} = \{ s_{1},s_{2} \}$
and
$\hat S_{2} = \{ s_{3},s_{4} \}$.
The minimum norm point (unique dec-min element)
of the base-polyhedron $B$ 
is given by $m_{\RR} = (3/2, 3/2, 1/2, 1/2 )$
by Theorem~\ref{THlexoptbaseR}.
Since
$\lceil \lambda_{1} \rceil = 2$ and
$\lceil \lambda_{2} \rceil = 1$, we have
$\beta_{1} = 2$ and 
$\beta_{2} = 1$ with $q=2$.
The canonical chain consists of two members 
$C_{1} = L(\beta_{1}-1)= L(1) = \{ s_{1},s_{2} \}$
and
$C_{2} = L(\beta_{2}-1)= L(0) = S$.
Accordingly,  the canonical partition is given by 
$S_{1} = \{ s_{1},s_{2} \}$
and
$S_{2} = \{ s_{3},s_{4} \}$.
Any of $m_{1}, \ldots, m_{5}$ 
is near-uniform on $S_{1}$ and on $S_{2}$,
but $m_{5}$ is not dec-min because it fails to satisfy
the tightness condition $\widetilde m(C_{1})=p(C_{1})$.
\finbox
\end{example}

\section{Proximity results and continuous relaxation algorithm}
\label{SCproximity}

In general terms, a proximity result means a statement
that the solutions to discrete and continuous versions
of an optimization problem are 
geometrically close to each other.
In this section we obtain proximity results
for dec-min elements $m_{\ZZ}$ of an M-convex set $\odotZ{B}$.
We shall establish two proximity theorems,
which refer to two different continuous problems.
The first proximity theorem refers to the (fractional) dec-min element $m_{\RR}$ of 
the base-polyhedron $B$,
and the second to the minimizer of a piecewise-linear function associated with the square-sum 
$W(x) = \sum [x(s)\sp{2}:  s\in S]$.
Both types of proximity results will be used as a basis for 
the continuous relaxation algorithm
to be described in Section~\ref{SCalgRZnorm}.

\subsection{Proximity theorem using the fractional dec-min element}

Our first proximity theorem reveals the
geometric closeness
of the dec-min element $m_{\ZZ}$ of $\odotZ{B}$
to the dec-min element $m_{\RR}$ of $B$.
By Theorem~\ref{THlexoptbaseR},
the dec-min element $m_{\RR}$ of $B$ is uniform 
on each member $\hat S_{i}$ of the principal partition
($m_{\RR}(s) = \lambda_{i}$ for $s \in  \hat S_{i}$),
whereas the dec-min element $m_{\ZZ}$ of $\odotZ{B}$ is near-uniform 
on each member $S_{j}$ of the canonical partition
($m_{\ZZ}(s) \in \{  \beta_{j}, \beta_{j} -1 \}$ for $s \in S_{j}$)
by Theorem~\ref{THchardecminZ}.
Combining these results with Theorem~\ref{THrelRZpartition} 
connecting the principal and canonical partitions, 
we can obtain the following proximity theorem.

\begin{theorem} \label{THdecminproxS}
Let $m_{\RR}$ be the dec-min element 
(minimum norm point)
of an integral base-polyhedron $B$.
Then every dec-min element $m_{\ZZ}$ of the associated M-convex set $\odotZ{B}$ satisfies
\begin{equation}  \label{mRmZmR}
\left\lfloor m_{\RR} \right\rfloor \leq m_{\ZZ} \leq  \left\lceil m_{\RR} \right\rceil.
\end{equation}
\end{theorem} 
\begin{proof}
Fix $s \in S$ and let $\hat S_{i}$ denote the member of the principal partition 
containing $s$, and $\lambda_{i}$ be the associated critical value.
We have $m_{\RR}(s) = \lambda_{i}$ by Theorem~\ref{THlexoptbaseR}.
By Theorem~\ref{THrelRZpartition},
$\lceil \lambda_{i} \rceil$
is an essential value, say,
$\lceil \lambda_{i} \rceil = \beta_{j}$, where $1 \leq j \leq q$.
Since the canonical partition is an aggregation of the principal partition,
the corresponding member $S_{j}$ of the canonical partition
contains the element $s$. 
We have $m_{\ZZ}(s) \in \{  \beta_{j}, \beta_{j} -1 \}$ 
by Theorem~\ref{THchardecminZ}.
Therefore, 
$m_{\ZZ} \leq  \left\lceil m_{\RR} \right\rceil$.

Next we apply the above argument to $-B$, which is also an integral base-polyhedron.
Since $-m_{\RR}$ is the minimum norm point of $-B$ and 
$-m_{\ZZ}$ is a dec-min (=inc-max) element for $-\odotZ{B}$, we obtain
$-m_{\ZZ} \leq  \left\lceil -m_{\RR} \right\rceil$, which is equivalent to  
$m_{\ZZ} \geq  \left\lfloor  m_{\RR} \right\rfloor$.
This completes the proof of \eqref{mRmZmR}.
\qedJIAM
\end{proof}

The above theorem states that
every dec-min element $m_{\ZZ}$ of $\odotZ{B}$ is located near 
the dec-min element $m_{\RR}$ of $B$, satisfying
$\left\lfloor m_{\RR} \right\rfloor \leq m_{\ZZ} \leq  \left\lceil m_{\RR} \right\rceil$.
However, the converse is not true, that is,
not every member $m$ of $\odotZ{B}$
satisfying 
$\left\lfloor m_{\RR} \right\rfloor \leq m \leq  \left\lceil m_{\RR} \right\rceil$
is a dec-min element of $\odotZ{B}$.
This is demonstrated by the following example.

\begin{example} \rm \label{EX4dimC4}
Recall Example~\ref{EX4dimC4pat},
where
$\odotZ{B}$ consists of five vectors:
$m_{1}=(2,1,1,0)$, \
$m_{2}=(2,1,0,1)$, \
$m_{3}=(1,2,1,0)$, \
$m_{4}=(1,2,0,1)$, and 
$m_{5}=(2,2,0,0)$.
The first four members, $m_{1}$ to $m_{4}$, 
are the dec-min elements of the M-convex set $\odotZ{B}$,
whereas 
$m_{5}=(2,2,0,0)$ is not dec-min.
The unique dec-min element of the (integral) base-polyhedron $B$
is $m_{\RR} = (3/2, 3/2, 1/2, 1/2 )$, for which 
$\left\lfloor m_{\RR} \right\rfloor = (1,1,0,0)$ and 
$\left\lceil m_{\RR} \right\rceil = (2,2,1,1)$.
Every member $m$ of $\odotZ{B}$ satisfies 
 $(1,1,0,0) = \left\lfloor m_{\RR} \right\rfloor
     \leq m \leq  \left\lceil m_{\RR} \right\rceil  = (2,2,1,1)$.
In particular, $m_{5}=(2,2,0,0)$
satisfies $\left\lfloor m_{\RR} \right\rfloor \leq m_{5} \leq  \left\lceil m_{\RR} \right\rceil$,
but it is not a dec-min element of $\odotZ{B}$.
\finbox
\end{example}

There is another connection 
between the dec-min elements in the continuous and discrete cases.
While Theorem \ref{THdecminproxS} above prescribes 
a region (box) for $m_{\ZZ}$ in terms of $m_{\RR}$,
the following theorem is a statement in the reverse direction,
showing that $m_{\RR}$
is embraced by the dec-min elements $m_{\ZZ}$
for the discrete case.

\begin{theorem} \label{THmnormconvcombdmS}
The unique dec-min element $m_{\RR}$ of an integral base-polyhedron $B$
can be represented as a convex combination of 
the dec-min elements $m_{\ZZ}$ of the associated M-convex set $\odotZ{B}$.
\end{theorem} 
\begin{proof}
It was shown in 
\cite[Section~5.1]{FM21partA}
that the dec-min elements of $\odotZ{B}$
lie on the face $B\sp{\oplus}$ of $B$ 
defined by the canonical chain 
$C_{1} \subset  C_{2} \subset  \cdots \subset  C_{q}$.
This face is the intersection of $B$ with the hyperplanes 
\[
\{x\in \RR\sp{S}:  \widetilde x(C_{j}) =  p(C_{j}) \}
\quad (j=1,2,\ldots, q).
\]
On the other hand, 
it is known (\cite{Fuj80}, \cite[Section~9.2]{Fuj05book}) that
the minimum norm point $m_{\RR}$ of $B$
lies on the face of $B$ 
defined by the principal chain 
$\hat C_{1} \subset  \hat C_{2} \subset  \cdots \subset \hat C_{r}$,
which is the intersection of $B$ with the hyperplanes 
\[
\{x\in \RR\sp{S}:  \widetilde x(\hat C_{i}) =  p(\hat C_{i}) \}
\quad (i=1,2,\ldots, r).
\]
Since the principal chain is a refinement of
the canonical chain (Theorem~\ref{THrelRZpartition}),
the latter face is a face of $B\sp{\oplus}$.
Therefore, 
$m_{\RR}$ belongs to $B\sp{\oplus}$. 
The point $m_{\RR}$ also belongs to 
\[
T\sp{*} := \{x\in \RR\sp{S}:  \ \beta_{j}-1\leq x(s)\leq \beta_{j} 
  \ \ \mbox{for}\ s\in S_{j} \ (j=1,2,\dots ,q)\},
\]
since
$m_{\RR}(s) = \lambda_{i}$ for $s \in  \hat S_{i}$
(Theorem~\ref{THlexoptbaseR}) and 
$ S_{j}
 = \bigcup \{ \hat S_{i}:   \lceil \lambda_{i} \rceil  = \beta_{j} \}$ 
(Theorem~\ref{THrelRZpartition}).
Therefore,  $m_{\RR}$ is a member of $B\sp{\bullet} := B\sp{\oplus}\cap T\sp{*}$.
Here $B\sp{\bullet}$ is an integral base-polyhedron, and 
Theorem~5.1 of \cite{FM21partA} states that 
the vertices of $B\sp{\bullet}$ are precisely the dec-min elements of $\odotZ{B}$.
Therefore, $m_{\RR}$
can be represented as a convex combination of the dec-min elements of $\odotZ{B}$.
\qedJIAM
\end{proof}

\begin{remark} \rm  \label{RMcnvcmbM2}
The convex combination property
stated in Theorem~\ref{THmnormconvcombdmS}
is no longer true for the intersection of two integral base-polyhedra. 
See \cite[Example 7.1]{FM21partA}.
\finbox
\end{remark}

\subsection{Proximity theorem using a piecewise-linear square-sum minimizer}

Our second proximity theorem shows the 
geometric closeness
of the dec-min element $m_{\ZZ}$ of $\odotZ{B}$
to the minimizer of a piecewise-linear function,
to be denoted by $\overline{W}(x)$,
arising from the square-sum
$W(x) = \sum [x(s)\sp{2}:  s\in S]$.
Recall from
Theorem~\ref{THdecminZsqsum} that an element of  $\odotZ{B}$
is dec-min if and only if it is a minimizer of 
$W(x)$ over $\odotZ{B}$.

To define the piecewise-linear function $\overline{W}(x)$, 
we first consider a piecewise-linear extension
of the quadratic function $\varphi (k)=k\sp{2}$ 
in a single integer variable $k \in \ZZ$.
The piecewise-linear extension
$\overline{\varphi}: \RR \to \RR$
is a function in a real variable
whose graph consists of line segments connecting
$(k,k\sp{2})$ and $(k+1,(k+1)\sp{2})$ 
for all $k \in \ZZ$.
That is,
\begin{equation} \label{pclinphi}
  \overline{\varphi}(\xi) := (2k+1) |\xi| - k (k+1) 
\quad \mbox{with \  $k = \lfloor |\xi| \rfloor$}
\qquad (\xi \in \RR).
\end{equation}
It is noted that 
$\overline{\varphi}(\xi) = \xi\sp{2}$ for integers $\xi$ and
$\overline{\varphi}(\xi) > \xi\sp{2}$ for non-integral $\xi$\,;
for example, $\overline{\varphi}(1/2) = 1/2 > 1/4$.
The piecewise-linear 
function $\overline{W}(x)$ 
is defined by
\begin{equation} \label{pclinW}
 \overline{W}(x) := \sum [ \overline{\varphi} ( x(s) ) : s \in S ]
\qquad (x \in \RR\sp{S}).
\end{equation}
We have
$\overline{W}(x) = W(x)$ for integral vectors $x$
and
$\overline{W}(x) > W(x)$ for non-integral vectors $x$.

The following fact is implicit in the proof of 
\cite[Theorem~8.3]{Fuj05book}.

\begin{proposition}  \label{PRcnvminGroe}
The minimum value of $\overline{W}$ over $B$
is equal to the minimum square-sum on the M-convex set $\odotZ{B}$.
Moreover, for any minimizer $x_{\RR} \in \RR\sp{S}$ of the function
$\overline{W}$ over $B$,
there exists a minimizer $x_{\ZZ}$ of 
$\overline{W}$ over $B$ satisfying
$x_{\ZZ} \in \ZZ\sp{S}$ and
$\lfloor x_{\RR} \rfloor \leq x_{\ZZ}  \leq  \lceil x_{\RR} \rceil$.
\end{proposition}
\begin{proof}
(This proof is essentially the same as the proof of \cite[Theorem~8.3]{Fuj05book}.)
Let $x_{\RR}$ be a minimizer of $\overline{W}$ over $B$,
and denote the intersection of $B$ with the box 
$ \{ x : \lfloor x_{\RR} \rfloor \leq x \leq  \lceil x_{\RR} \rceil\}$
by $B'$, which is also an integral base-polyhedron.
Since 
$\| \, \lceil x_{\RR} \rceil - \lfloor x_{\RR} \rfloor \, \|_{\infty} \leq 1$,
the function 
$\overline{W}$ is linear on $B'$.
This implies that
$x_{\RR}$ can be expressed as a convex combination 
of some vertices 
$z_{1}, z_{2},  \allowbreak  \ldots,  \allowbreak    z_{k}$ 
of $B'$
and the function value $\overline{W}(x_{\RR})$
is given by the corresponding convex combination 
of their function values.
That is,
\[
x_{\RR} = \sum_{i=1}\sp{k} \alpha_{i} z_{i},
\quad
\overline{W}(x_{\RR}) 
= \sum_{i=1}\sp{k} \alpha_{i} \overline{W}(z_{i}) ,
\]
where $\sum_{i=1}\sp{k} \alpha_{i} = 1$ and $\alpha_{i} > 0$ for all $i$.
Since $x_{\RR}$ is a minimizer of $\overline{W}$, we have
$\overline{W}(x_{\RR}) \leq  \overline{W}(z_{i})$ for all $i$,
and hence each $z_{i}$ is a minimizer of $\overline{W}$ over $B$,
for which $\overline{W}(x_{\RR}) = \overline{W}(z_{i})$.
Moreover, 
$z_{i}$ is an integral vector satisfying 
$\lfloor x_{\RR} \rfloor \leq z_{i}  \leq  \lceil x_{\RR} \rceil$.
Therefore, we can take any $z_{i}$ as $x_{\ZZ}$.
\qedJIAM
\end{proof}

By combining Proposition~\ref{PRcnvminGroe}
with Theorem \ref{THdecminZsqsum} 
(characterizing dec-min elements of $\odotZ{B}$
as square-sum minimizers), we obtain 
the following proximity statement.

\begin{theorem} \label{THdecminproxG}
For any minimizer $x_{\RR}$ 
of the function $\overline{W}$ over $B$, there exists 
a dec-min element $m_{\ZZ}$ of the associated M-convex set $\odotZ{B}$ satisfying
\begin{equation} \label{decminproxW}
\left\lfloor x_{\RR} \right\rfloor \leq m_{\ZZ} \leq  \left\lceil x_{\RR} \right\rceil .
\end{equation}
\end{theorem}
\begin{proof}
By Proposition~\ref{PRcnvminGroe}, 
there exists an integral vector $x_{\ZZ}$
that minimizes $\overline{W}$ over $B$
and satisfies 
$\lfloor x_{\RR} \rfloor \leq x_{\ZZ}  \leq  \lceil x_{\RR} \rceil$.
Since $\overline{W}$ coincides with $W$ on $\odotZ{B}$,
this vector $x_{\ZZ}$ is a minimizer of $W$ over $\odotZ{B}$.
This implies, by Theorem \ref{THdecminZsqsum},
that $x_{\ZZ}$ is a dec-min element of $\odotZ{B}$
satisfying \eqref{decminproxW}.
Therefore, we can take this $x_{\ZZ}$ as $m_{\ZZ}$
in \eqref{decminproxW}.
\qedJIAM
\end{proof}

There are substantial differences 
between the two proximity results given in 
Theorem~\ref{THdecminproxG} and in Theorem~\ref{THdecminproxS}.
First, the vector $x_{\RR}$ 
in Theorem~\ref{THdecminproxG},
being an arbitrary minimizer of $\overline{W}$,
is not uniquely determined, whereas $m_{\RR}$ 
in Theorem~\ref{THdecminproxS}
denotes the unique dec-min element of $B$.
In particular, $x_{\RR}$ is not necessarily dec-min in $B$
(see Example \ref{EX4dimC4groen} below).
Second, the box
$\left\lfloor x_{\RR} \right\rfloor \leq x \leq  \left\lceil x_{\RR} \right\rceil$
in \eqref{decminproxW}
of Theorem~\ref{THdecminproxG}
may possibly miss some dec-min elements of $\odotZ{B}$
(see Example \ref{EX4dimC4groen}),
while the box 
$\left\lfloor m_{\RR} \right\rfloor \leq x \leq  \left\lceil m_{\RR} \right\rceil$
in \eqref{mRmZmR} of Theorem~\ref{THdecminproxS}
contains all dec-min elements of $\odotZ{B}$.

\begin{example} \rm \label{EX4dimC4groen}
We continue with the problem 
treated in Examples \ref{EX4dimC4pat} and \ref{EX4dimC4}.
The M-convex set $\odotZ{B}$ consists of five vectors:
$m_{1}=(2,1,1,0)$, \
$m_{2}=(2,1,0,1)$, \
$m_{3}=(1,2,1,0)$, \
$m_{4}=(1,2,0,1)$, and 
$m_{5}=(2,2,0,0)$.
We have $W(m_{i})= 6$ for $i=1,2,3,4$ and $W(m_{5})= 8$. 
Hence the minimum of $\overline{W}$ over $B$ is equal to 6.
Consider a vector $x_{\RR}=(2,1,1/3,2/3)$,
which is a minimizer of $\overline{W}(x)$ since
\[
\overline{W}(x_{\RR}) =
\overline{\varphi}(2)
+\overline{\varphi}(1)
+\overline{\varphi}(1/3)
+\overline{\varphi}(2/3) = 4 + 1 + 1/3 + 2/3 = 6.
\]
For this vector, we have
$\lfloor x_{\RR} \rfloor =  (2,1,0,0)$ and 
$\lceil x_{\RR} \rceil  = (2,1,1,1)$.
The box
$\left\lfloor x_{\RR} \right\rfloor \leq x \leq  \left\lceil x_{\RR} \right\rceil$
contains
$m_{1}=(2,1,1,0)$ and $m_{2}=(2,1,0,1)$,
but misses the other two dec-min elements, $m_{3}$ and $m_{4}$, of $\odotZ{B}$.
Another possible choice of a minimizer of $\overline{W}$
is  the minimum norm point $m_{\RR} = (3/2, 3/2, 1/2, 1/2 )$,
which is indeed a minimizer of $\overline{W}$ 
(see Proposition~\ref{PRminnormminPLsqsum} below)
with
\[
\overline{W}(m_{\RR}) =
\overline{\varphi}(3/2)
+\overline{\varphi}(3/2)
+\overline{\varphi}(1/2)
+\overline{\varphi}(1/2) = 5/2 + 5/2 + 1/2 + 1/2 = 6 .
\]
For this vector, we have
$\lfloor m_{\RR} \rfloor =  (1,1,0,0)$ and
$\lceil m_{\RR} \rceil  = (2,2,1,1)$.
The box
$\left\lfloor m_{\RR} \right\rfloor \leq x \leq  \left\lceil m_{\RR} \right\rceil$
contains all the four dec-min elements, and additionally, 
$m_{5}=(2,2,0,0)$ (which is not dec-min). 
\finbox
\end{example}

We point out here that
the minimum norm point $m_{\RR}$, which is
the unique minimizer of the square-sum $W(x)$ over $B$, 
is also a minimizer of the associated piecewise-linear function
$\overline{W}$.
This fact is quite natural to expect, but it is a non-trivial fact
whose proof
relies on the property (Theorem~\ref{THmnormconvcombdmS})
that $m_{\RR}$ lies in the convex hull of the dec-min elements of $\odotZ{B}$.

\begin{proposition} \label{PRminnormminPLsqsum}
The minimum norm point $m_{\RR}$ of $B$
is a minimizer of the piecewise-linear function $\overline{W}$ on $B$.
\end{proposition}
\begin{proof}
Theorem~\ref{THmnormconvcombdmS} ensures a convex combination 
$m_{\RR}$ $ = \sum_{i=1}\sp{k} \alpha_{i} m_{i}$,
where
$m_{1},m_{2}, \ldots, m_{k}$ are dec-min elements of $\odotZ{B}$,
$\sum_{i=1}\sp{k} \alpha_{i} = 1$, and $\alpha_{i} > 0$ for all $i$.
Since 
$\| m_{i} - m_{j}  \|_{\infty} \leq 1$,
the function 
$\overline{W}$ is linear on 
the convex hull of 
$m_{1},m_{2}, \ldots, m_{k}$,
from which follows that 
\[
\overline{W}(m_{\RR}) 
= \sum_{i=1}\sp{k} \alpha_{i} \overline{W}(m_{i})
= \sum_{i=1}\sp{k} \alpha_{i} W(m_{i}).
\]
Here each of
$m_{1},m_{2}, \ldots, m_{k}$
is a minimizer of $W$ (by Theorem \ref{THdecminZsqsum}) and
$\min \{ \overline{W}(x) : x \in B \} 
= \min \{ W(x) : z \in \odotZ{B} \}$
by the definition of $\overline{W}$
and the integrality of $B$.
Therefore, $m_{\RR}$ is a minimizer of $\overline{W}$.
\qedJIAM
\end{proof}

\subsection{Continuous relaxation algorithm}
\label{SCalgRZnorm}

In our continuous relaxation algorithm
for computing a dec-min element of $\odotZ{B}$,
it is assumed that we are given a real vector 
$x\sp{*}$ such that
the box bounded by 
$\left\lfloor x\sp{*} \right\rfloor$ and $\left\lceil x\sp{*} \right\rceil$
contains at least one dec-min element of $\odotZ{B}$.
That is, we assume that 
\begin{equation} \label{decminproxA}
\left\lfloor x\sp{*} \right\rfloor \leq m_{\ZZ} \leq  \left\lceil x\sp{*} \right\rceil 
\end{equation}
holds for some dec-min element $m_{\ZZ}$ of $\odotZ{B}$.
By Theorem~\ref{THdecminproxS} 
the minimum norm point $m_{\RR}$ of $B$
serves as such $x\sp{*}$.
Another possibility for $x\sp{*}$ is 
an arbitrary minimizer $x_{\RR}$ of the piecewise-linear function $\overline{W}(x)$,
as shown in Theorem~\ref{THdecminproxG}.
The algorithm of this section relies only on the property
\eqref{decminproxA}
of the vector $x\sp{*}$,
which is regarded as an input of the algorithm.
At the end of this section, we indicate
some references concerning the computation of $m_{\RR}$ and $x_{\RR}$
in Remarks \ref{RMminnormcomp} and \ref{RMcontrelalgKSI}, respectively.

Using the given vector $x\sp{*}$ satisfying \eqref{decminproxA}, define 
\begin{equation}  \label{ludefnorm}
\ell := \left\lfloor x\sp{*} \right\rfloor,
\quad
u := \left\lceil x\sp{*} \right\rceil ,
\end{equation}
and denote the intersection of $\odotZ{B}$ and box $[\ell, u]$ 
by $\odotZ{B_{\ell}\sp{u}}$,
that is,
\begin{equation}  \label{dotBul}
 \odotZ{B_{\ell}\sp{u}} :=  \odotZ{B} \cap \{ x :   \ell \leq x \leq u \} .
\end{equation}
By our assumption, $\odotZ{B_{\ell}\sp{u}}$ contains at least one 
dec-min element of $\odotZ{B}$.
This implies that 
we can find a dec-min element of $\odotZ{B}$ 
by computing a dec-min element of $\odotZ{B_{\ell}\sp{u}}$.
Thus the dec-min problem on the M-convex set $\odotZ{B}$
is reduced to that on a smaller M-convex set $\odotZ{B_{\ell}\sp{u}}$.

A dec-min element of $\odotZ{B_{\ell}\sp{u}}$ can be computed as follows.
Since $\bm{0} \leq u - \ell \leq \bm{1}$,
the set $\odotZ{B_{\ell}\sp{u}}$ 
is a matroidal M-convex set, and hence it can be represented as
\[
\odotZ{B_{\ell}\sp{u}} = \{ \ell + \chi_{L}: 
\mbox{$L$ is a base of $M\sp{\bullet}$} \}
\]
for some matroid $M\sp{\bullet}$ on $S$.
Define a weight function 
$\omega: S \to \ZZ$ by
\begin{equation} \label{weightdef}
 \omega(s) := u(s)\sp{2} - \ell(s)\sp{2}
 \quad (s \in S).
\end{equation}
Then the square-sum $W(x)$ of 
$x = \ell + \chi_{L} \in \odotZ{B_{\ell}\sp{u}}$
can be expressed as
\begin{equation*} 
W(x) = \sum_{s \in S} x(s)\sp{2} 
 = \sum_{s \in L} u(s)\sp{2} + \sum_{s \in S-L} \ell(s)\sp{2}  
 = \widetilde \omega(L) + \sum_{s \in S} \ell(s)\sp{2}  .
\end{equation*}
This shows that minimizing $W(x)$ 
over $\odotZ{B_{\ell}\sp{u}}$
is equivalent to finding a minimum $\omega$-weight base of $M\sp{\bullet}$,
whereas, by Theorem~\ref{THdecminZsqsum},
an element of  $\odotZ{B_{\ell}\sp{u}}$
is a minimizer of $W(x)$ over $\odotZ{B_{\ell}\sp{u}}$
if and only if it is dec-min in $\odotZ{B_{\ell}\sp{u}}$.
Therefore, a dec-min element of $\odotZ{B_{\ell}\sp{u}}$ 
can be computed by finding a minimum $\omega$-weight base of matroid $M\sp{\bullet}$.
The latter can be done in strongly polynomial time 
by the greedy algorithm (see, e.g., \cite{Fra11book,Sch03}).
In order to apply the greedy algorithm, one needs an evaluation oracle that
outputs the rank $r\sp{\bullet}(Z)$ of any input subset $Z\subseteq S$
in matroid $M\sp{\bullet}$.
But $r\sp{\bullet}(Z)$ can be computed from the supermodular function
$p$ associated with $\odotZ{B}$ and from the bounding vectors $\ell$
and $u$ defined in \eqref{ludefnorm} 
with the aid of a submodular function minimization algorithm.
Therefore, the above procedure, when $x\sp{*}$ is given,
finds a dec-min element of $\odotZ{B}$ 
in strongly polynomial time.

\begin{example} \rm \label{EX4dimC4relax}
We illustrate the continuous relaxation algorithm for 
the problem considered in Examples~\ref{EX4dimC4pat}, \ref{EX4dimC4}, and \ref{EX4dimC4groen},
where
$\odotZ{B}$ consists of five vectors:
$m_{1}=(2,1,1,0)$, \
$m_{2}=(2,1,0,1)$, \
$m_{3}=(1,2,1,0)$, \
$m_{4}=(1,2,0,1)$, and 
$m_{5}=(2,2,0,0)$;
$m_{k}$ is dec-min for $k=1,2,3,4$, while $m_{5}$ is not.

Suppose first that 
the minimum norm point $m_{\RR} = (3/2, 3/2, 1/2, 1/2 )$
is chosen as $x\sp{*}$.
Then we obtain 
$\ell = \lfloor m_{\RR} \rfloor = (1,1,0,0)$ and 
$u = \lceil m_{\RR} \rceil =(2,2,1,1)$, and hence 
$\odotZ{B_{\ell}\sp{u}} = \odotZ{B}$ and
$\omega = (3,3,1,1)$. 
For $x = m_{1}$
we have
$m_{1} - \ell = (2,1,1,0) - (1,1,0,0) = (1,0,1,0) = \chi_{L_{1}}$
with $L_{1} = \{ s_{1}, s_{3} \}$ and 
$\widetilde \omega(L_{1}) = 3 + 0 + 1 + 0 = 4$.
Similarly, we have
$m_{k} - \ell = \chi_{L_{k}}$ with $\widetilde \omega(L_{k}) = 4$ for $k=2,3,4$.
For $x = m_{5}$  we have
$m_{5} - \ell = (2,2,0,0) - (1,1,0,0) = (1,1,0,0) = \chi_{L_{5}}$
with $L_{5} = \{ s_{1}, s_{2} \}$ and 
$\widetilde \omega(L_{5}) = 3 + 3 + 0 + 0 = 6$.
Therefore, 
$L_{k}$ is a minimum $\omega$-weight base for $k=1,2,3,4$,
while $L_{5}$ is not.
In other words, $m_{k}$ is dec-min in $\odotZ{B}$
for $k=1,2,3,4$, while $m_{5}$ is not.

As the second choice of $x\sp{*}$,
consider a minimizer $x_{\RR}=(2,1,1/3,2/3)$ of $\overline{W}(x)$,
for which 
$\ell = \lfloor x_{\RR} \rfloor =  (2,1,0,0)$,
$u = \lceil x_{\RR} \rceil  = (2,1,1,1)$,
and hence
$\odotZ{B_{\ell}\sp{u}} = \{ m_{1}, m_{2} \}$
and $\omega = (0,0,1,1)$. 
We have
$m_{1} - \ell = (0,0,1,0) = \chi_{L_{1}}$
with $L_{1} = \{  s_{3} \}$ and 
$\widetilde \omega(L_{1}) = 1$,
while
$m_{2} - \ell = (0,0,0,1) = \chi_{L_{2}}$
with $L_{2} = \{  s_{4} \}$ and 
$\widetilde \omega(L_{2}) = 1$.
Therefore, we can conclude that both 
$m_{1}$ and $m_{2}$
are dec-min in $\odotZ{B}$.
The other two dec-min elements of $\odotZ{B}$, 
$m_{3}$ and $m_{4}$, are not captured 
when $x\sp{*} = x_{\RR}=(2,1,1/3,2/3)$.
\finbox
\end{example}

\begin{remark} \rm  \label{RMcheapdmprox} 
The continuous relaxation algorithm using the minimum norm point $m_{\RR}$
can cope with the minimum cost dec-min element problem
(Section \ref{SCcmparUni}),
since all dec-min elements of $\odotZ{B}$ 
are captured by $\odotZ{B_{\ell}\sp{u}}$
(Theorem~\ref{THdecminproxS}).
In contrast, the continuous relaxation algorithm 
using a minimizer $x_{\RR}$ of $\overline{W}(x)$
cannot be used to solve this problem.
\finbox
\end{remark}

\begin{remark} \rm  \label{RMminnormcomp}
Several different algorithms are known for computing 
the minimum norm point $m_{\RR}$,
with varying theoretical complexity and practical efficiency. 
Fujishige's decomposition algorithm \cite{Fuj80}
(also \cite[Section~8.2]{Fuj05book})
computes $m_{\RR}$ in  strongly polynomial time.
We can also compute $m_{\RR}$ by 
Wolfe's minimum norm point algorithm \cite{Wol76}
 tailored to base-polyhedra \cite[Section~7.1]{Fuj05book},
which algorithm has drawn a renewed interest 
as a practically effective subroutine in submodular function minimization
(Chakrabarty--Jain--Kothari \cite{CJK14}, 
De Loera--Haddock--Rademacher \cite{DLHR20}, 
Fujishige--Isotani \cite{FI11minnorm}).
\finbox
\end{remark}

\begin{remark}\rm  \label{RMcontrelalgKSI} 
Minimization of a separable convex function on a base-polyhedron 
has been investigated in the literature of resource allocation
under the name of ``resource allocation problems under submodular constraints''
(Hochbaum \cite{Hoc07}, Ibaraki--Katoh \cite{IK88}, 
Katoh--Shioura--Ibaraki \cite{KSI13}).
The continuous relaxation approach for discrete variables
is considered, e.g., by Hochbaum \cite{Hoc94} and Hochbaum--Hong \cite{HH95}.
A paper by Moriguchi--Shioura--Tsuchimura \cite{MST11Mrelax}
discusses this approach in the more general context of M-convex function minimization.
It is known (\cite{HH95,MST11Mrelax,Tam93res}, \cite[Theorem 23]{KSI13})
that a convex quadratic function
$\sum a_{i} x_{i}\sp{2}$ 
in discrete variables can be minimized over an integral base-polyhedron
in strongly polynomial time.
See  V{\'e}gh \cite{Veg16} for a recent development 
on the complexity of separable convex function minimization.
\finbox
\end{remark}

\section{Decomposition algorithms for square-sum minimization on an M-convex set}
\label{SCdecalg}

\subsection{General remarks}
\label{SCdecalgGen}

As a continuation from Section~\ref{SCcmparAlg},
we shall present details of the decomposition algorithms
for square-sum minimization on an M-convex set.
Recall from Theorem~\ref{THdecminZsqsum} that
minimizing the square-sum 
$W(x) = \sum [x(s)\sp{2}:  s\in S]$
on an M-convex set $\odotZ{B}$
is equivalent to 
computing a dec-min element of $\odotZ{B}$.

In order to present a clear overview of the existing approaches,
we consider two variants of 
the decomposition algorithms,
one based on Groenevelt \cite{Gro91} and the other based on
Fujishige \cite[Section~8.2]{Fuj05book}
(see Table~\ref{TBhistoryDA}).
Although the same name of ``decomposition algorithm'' is used,
Groenevelt's and Fujishige's are not exactly the same.
To be specific,  Groenevelt's algorithm features a subproblem named
``single constraint problem'' and 
decomposes the ground-set into two disjoint subsets,
whereas
Fujishige's algorithm uses a variable 
corresponding to a subgradient 
(denoted $\eta$ in \cite[Section~8.2]{Fuj05book}) 
and decomposes the ground-set into three disjoint subsets.
It is noted that Fujishige's original algorithm \cite{Fuj80},
targeted to quadratic functions in continuous variables, was based 
on a simplest special case of
``single constraint problem'' (without using this terminology)
and decomposes the ground-set into two disjoint subsets.

In Section~\ref{SCdecalgG}, 
we show an adaptation of Groenevelt's decomposition algorithm 
to square-sum minimization on an M-convex set,
and call it the ``Groenevelt-type decomposition algorithm.''
In its original form, 
Groenevelt's algorithm \cite{Gro91} for Case~$\ZZ$
dealt with an integral polymatroid (not an integral base-polyhedron).
Although the adaptation to an M-convex set is not that difficult, 
it will be important to have a precise description of the algorithm at hand
along with proofs of correctness and strong polynomiality 
of the algorithm.

In Section~\ref{SCdecalgF}, 
we develop another decomposition algorithm 
for square-sum minimization on an M-convex set
on the basis of the framework of Fujishige \cite[Section~8.2]{Fuj05book}
for separable convex minimization on a base-polyhedron.
We refer to the resulting algorithm as the ``Fujishige-type decomposition algorithm.''
As suggested in \cite[Section~8.3]{Fuj05book},
the framework of \cite[Section~8.2]{Fuj05book} for Case~$\RR$ 
can be adapted to Case~$\ZZ$ with the use of the piecewise-linear extension
$\overline{W}(x)$ in \eqref{pclinW}.
Our contribution consists in devising a concrete computational procedure
for a key subroutine assumed in the general framework,
giving self-contained rigorous proofs
 of correctness and  strong polynomiality of the algorithm,
and revealing the relation between the computed decomposition 
and the canonical chain
as well as the certifying chain in Theorem \ref{THdecmincharZ}.

In relation to the structural and algorithmic results in \cite{FM21partA,FM21partB},
the two decomposition algorithms, 
Groenevelt-type and Fujishige-type,
have the following characteristics.

\begin{itemize} 
\item
Unlike the algorithm in \cite{FM21partB},
the two decomposition algorithms
do not rely on the Newton--Dinkelbach algorithm.

\item
The two decomposition algorithms
decompose the ground-set, but the resulting decomposition 
may or may not coincide with the canonical partition,
whereas the algorithm in \cite{FM21partB} iteratively construct
the canonical partition.
It should be noted, however, that
if a single dec-min element of an M-convex set is available, 
the canonical partition can be computed quite easily \cite[Algorithm 2.3]{FM21partB}.

\item
The Groenevelt-type algorithm is simpler than the Fujishige-type,
both in steps of the algorithm and the proof of correctness,
and is rather independent of the structural results found in \cite{FM21partA}.
In contrast, the decomposition computed in  the Fujishige-type algorithm
is consistent with a characterization of dec-minimality given  
in \cite{FM21partA}, as pointed out in Remark~\ref{RMdecalgFpat}.
\end{itemize}

\subsection{Groenevelt-type decomposition algorithm}
\label{SCdecalgG}

In this section, we show an adaptation of Groenevelt's decomposition algorithm 
to square-sum minimization on an M-convex set,
which we call ``Groenevelt-type decomposition algorithm'' in this paper.

Let $B = B'(p)$ be an integral base-polyhedron on a ground-set $S$
described by an integer-valued supermodular function $p$,
and $\odotZ{B} = B \cap \ZZ\sp{S}$ be the associated M-convex set.
For any subset $S_{+}$ of $S$,
the restriction of $B$ to $S_{+}$ 
means the base-polyhedron $B_{+} := B'(p_{+})$ 
described by the supermodular function $p_{+}$ defined by
$p_{+}(X) = p(X)$ for $X  \subseteq S_{+}$.
For any subset $S_{-}$ of $S$,
the contraction of $B$ to $S_{-}$ 
is the base-polyhedron $B_{-} := B'(p_{-})$ 
described by the supermodular function $p_{-}$ defined by
$p_{-}(X) = p(X \cup (S - S_{-})) - p(S - S_{-})$ for $X  \subseteq S_{-}$.
We also define the supermodular polyhedron
\begin{equation} \label{grosQdef}
Q = \{ y \in \RR\sp{S} : y(X) \geq p(X) \ (\forall X \subseteq S)  \},
\end{equation}
which can also be defined as 
$Q = \{ y \in \RR\sp{S} : y \geq x  \ \mbox{ for some $x \in B$}  \}$.

\medskip

Groenevelt's framework employs an auxiliary subproblem,
called ``single constraint problem.''
In our case of square-sum minimization, this subproblem is given by
\begin{equation} \label{grosnglM}
 \mbox{Minimize} \ \   \sum_{s \in S} x(s)\sp{2}
\quad \mbox{ subject to} \quad
x \in \ZZ\sp{S}, \  x(S) = p(S) ,
\end{equation}
which (fortunately) admits an explicit solution.
Let  
\begin{equation} \label{grosnglMak}
a := \lfloor p(S) / |S| \rfloor,  \qquad
k := p(S) - a |S|,
\end{equation}
where $0 \leq k \leq |S|-1$.
Then a vector $x$ is a solution to 
the problem \eqref{grosnglM}
if and only if 
$x(s) \in \{ a, a+1 \}$ for all $s \in S$
and precisely $k$ of $x(s)$ are equal to $a+1$;
that is, 
\begin{equation} \label{grosnglMsol}
x = a \chi_{S} + \chi_{U}
\end{equation}
for any $U \subset S$ with $|U| = k$.
For example,
$x = (\overbrace{a+1, \ldots, a+1}^{k}, a, \ldots, a)$.

With above preparations we can describe the algorithm
to minimize the square-sum $W(x)$ on an M-convex set.


\medskip

\noindent
{\bf Groenevelt-type decomposition algorithm for square-sum on an M-convex set $\odotZ{B}$}%
\\
1:  
   Let $a := \lfloor p(S) / |S| \rfloor$, \  $k := p(S) - a |S|$,  
 and  $x := a \chi_{S} + \chi_{U}$
  for an arbitrary $U$ with $|U| = k$.
\\
2:
    If $x$ belongs to $\odotZ{B}$, then let $z\sp{*}:=x$ and stop. 
\\
3: 
     Find a minimal vector $y$ satisfying $y \geq x$ and $y \in \odotZ{Q}$.
\\
4: 
    Let $S_{+}$ be the largest subset $X \subseteq S$ satisfying $y(X) = p(X)$,
      and let $S_{-} := S \setminus S_{+}$.
\\
5: 
    Apply this algorithm recursively to the 
     restriction $\odotZ{B_{+}}$ to $S_{+}$.
    Let  $z\sp{*}_{+} \in \ZZ\sp{S_{+}}$ be the output.
\\
6: 
   Apply this algorithm recursively to the 
     contraction $\odotZ{B_{-}}$ to $S_{-}$.
    Let  $z\sp{*}_{-} \in \ZZ\sp{S_{-}}$ be the output.
\\
7:  
   Define $z\sp{*}(s) := z\sp{*}_{+}(s)$ for $s \in S_{+}$ and
   $z\sp{*}(s) := z\sp{*}_{-}(s)$ for $s \in S_{-}$, and stop.


\medskip

Basic properties of the above algorithm are given below.
The second property (2) implies 
that the recursive calls in Step~5 and Step~6 make sense.

\begin{proposition} \label{PRgroebasic}
\quad
\\
\noindent
{\rm (1)} \ 
When the algorithm terminates in Step~2, the output $z\sp{*}$ is a square-sum minimizer.
\\
\noindent
{\rm (2)} \ 
In Step~4, we have 
$\emptyset \neq S_{+} \neq S$ and  $\emptyset \neq S_{-} \neq S$.
\\
\noindent
{\rm (3)} \ 
$z\sp{*}$ defined in Step~7 is a member of $\odotZ{B}$.
\\
\noindent
{\rm (4)} \ 
$y(s) = x(s)$ for all $s \in S-S_{+}$.
\end{proposition} 

\begin{proof}
(1)  
$z\sp{*}$ is near-uniform at the termination in Step~2, and hence it is decreasingly minimal, or equivalently, a square-sum minimizer.

(2) 
We have $y \geq x$, $y \neq x$, $x \notin \odotZ{B}$, and $y \in \odotZ{Q}$ in Step~4.
Since $y$ is not equal to $x$, there exists an element $s \in S$ 
with $y(s) > x(s)$.
Then the minimality of $y$ implies that such $s$ must
belong to $S_{+}$, implying $S_{+} \neq \emptyset$. 
Because of such $s$, we have
$\widetilde y(S) > \widetilde x(S) = p(S)$,
which shows $S_{+} \neq S$.

(3)
This is immediate from the fact that 
$z\sp{*}| S_{+} = z\sp{*}_{+} \in \odotZ{B_{+}}$
and $z\sp{*}| S_{-} =z\sp{*}_{-} \in \odotZ{B_{-}}$,
where $\odotZ{B_{+}}$ is the restriction of $\odotZ{B}$ to $S_{+}$ and
$\odotZ{B_{-}}$ is the contraction of $\odotZ{B}$ to $S_{-}$.

(4) 
By definition, $y$ is a minimal element of $\odotZ{Q}$ satisfying $y \geq x$,
and $S_{+}$ is the largest $y$-tight set with respect to $p$.
Suppose, indirectly, that $y(s) > x(s)$ for some $s \in S-S_{+}$.
Define $\hat y := y - \chi_{s}$. This $\hat y$ belongs to $\odotZ{Q}$, since $y(X) > p(X)$
for any $X$ containing $s$. Also we have $\hat y \geq x$,
a contradiction to the minimality of $y$.
\qedJIAM
\end{proof}

The correctness of the algorithm is established in the following proposition.

\begin{proposition} \label{PRdecalgGoutput}
The output $z\sp{*}$ of the Groenevelt-type decomposition algorithm 
is a minimizer of square-sum $W(x)$ over $\odotZ{B}$,
or equivalently, $z\sp{*}$ is a dec-min element of $\odotZ{B}$.
\end{proposition}
\begin{proof}
By induction on the size of the ground-set $S$,
we prove that the output of the algorithm is a square-sum minimizer.
By Proposition~\ref{PRgroebasic}(2), we have $|S_{+}| < |S|$  and $|S_{-}| < |S|$.
First we show that the output $z\sp{*}$ of the algorithm satisfies
\begin{equation}
 z\sp{*}(s)  \geq a  \quad (s \in S_{+} ) ,
\qquad
 z\sp{*}(s) \leq a+1 \quad  (s \in S_{-} ) .
\label{decalgz*valGr}
\end{equation}

To prove the first inequality
$z\sp{*}(s) \geq a$ $(s \in S_{+})$
in \eqref{decalgz*valGr},
let $z_{+}\sp{*}$ denote the subvector of $z\sp{*}$ on $S_{+}$
(i.e., $z_{+}\sp{*} = z\sp{*}|S_{+}$),
which is the outcome of the recursive call to 
the restriction $\odotZ{B_{+}}$ to $S_{+}$.
By the induction hypothesis,
$z_{+}\sp{*}$ minimizes 
the square-sum
$\sum [ x(s)\sp{2}: s \in S_{+} ]$
over $\odotZ{B_{+}}$.
This implies, 
by Theorems \ref{THdecincZ} and \ref{THdecminZsqsum}, 
that $z_{+}\sp{*}$ is an inc-max element of $\odotZ{B_{+}}$.
On the other hand, $y|S_{+}$ is a member of $\odotZ{B_{+}}$
(since $S_{+}$ is $y$-tight with respect to $p$) 
satisfying 
$(y|S_{+})(s) \geq x(s) \geq a$ for all $s \in S_{+}$.
It then follows that each component of $z_{+}\sp{*}$  
is bounded from below by $a$.
Therefore, $z\sp{*}(s) \geq a $ for $s \in S_{+}$.

The second inequality $z\sp{*}(s) \leq a+1$ $(s \in S_{-})$
in \eqref{decalgz*valGr}
can be proved as follows.
By the induction hypothesis,
the vector $z_{-}\sp{*} := z\sp{*}|S_{-}$
minimizes the square-sum
over the contraction $\odotZ{B_{-}}$ to $S_{-}$,
and hence $z_{-}\sp{*}$ is a dec-min element of $\odotZ{B_{-}}$
by Theorem~\ref{THdecminZsqsum},
which further implies that 
$z_{-}\sp{*}$ is a dec-min element of 
$\odotZ{Q_{-}} = \{ v \in \ZZ\sp{S_{-}} : v \geq u  \ \mbox{ for some $u \in \odotZ{B_{-}}$}  \}$.
On the other hand, $y|S_{-}$ is a member of $\odotZ{Q_{-}}$
satisfying 
$(y|S_{-})(s) \leq a+1$ for all $s \in S_{-}$,
since $y(s) = x(s) \leq a+1$ for all $s \in S_{-}$
by Proposition~\ref{PRgroebasic}(4)
and \eqref{grosnglMsol}. 
Therefore
$z\sp{*}(s) \leq a+1$ for all $s \in S_{-}$.
Thus \eqref{decalgz*valGr} is proved.

Finally we show that $z\sp{*}$ satisfies the condition 
\begin{equation} \label{decminoptexcZGr}
 z\sp{*}(t) \geq z\sp{*}(s) + 2
\ \Longrightarrow \ 
 z\sp{*} + \chi_{s} - \chi_{t} \notin \odotZ{B} 
\end{equation}
for all $(s,t)$ with $s, t \in S_{+} \cup S_{-}$.
Suppose that $z\sp{*}(t) \geq z\sp{*}(s) + 2$.
By \eqref{decalgz*valGr} we may assume
$(s,t) \notin S_{+} \times S_{-}$.
The condition \eqref{decminoptexcZGr} holds 
when $(s,t) \in S_{+} \times S_{+}$,
since the vector $z_{+}\sp{*} = z\sp{*}|S_{+}$ 
is a dec-min element of $\odotZ{B_{+}}$.
Similarly, \eqref{decminoptexcZGr} holds 
when $(s,t) \in S_{-} \times S_{-}$,
since the vector $z_{-}\sp{*} = z\sp{*}|S_{-}$ 
is a dec-min element of $\odotZ{B_{-}}$.
It remains to consider the case of
$(s,t) \in S_{-} \times S_{+}$.
Since $z_{+}\sp{*} \in \odotZ{B_{+}}$,
the set $S_{+}$ is $z\sp{*}$-tight with respect to $p$.
It then follows that 
$z\sp{*} + \chi_{s} - \chi_{t} \notin \odotZ{B}$
if $s \in S - S_{+}$ and $t \in S_{+}$.
Hence, the condition \eqref{decminoptexcZGr} holds for 
$(s,t) \in S_{-}  \times S_{+}$.
By Theorems \ref{THdecmincharZ} and \ref{THdecminZsqsum}, 
this completes the proof of Proposition~\ref{PRdecalgGoutput}.
\qedJIAM
\end{proof}

For clarity we make an explicit statement about strong polynomiality of the algorithm.

\begin{proposition} \label{PRdecalgGpolytime}
The Groenevelt-type decomposition algorithm
computes a square-sum minimizer (i.e., dec-min element) of an M-convex set
in strongly polynomial time.
\end{proposition}
\begin{proof}
The number of recursive calls is bounded by $|S|$,
and each step  can be done in strongly polynomial time
(using a submodular function minimization subroutine).
\qedJIAM
\end{proof}

The algorithm is illustrated for a simple example.

\begin{example} \rm \label{EXgroedim2}
Let $S = \{ s_{1} , s_{2} \}$ and $\odotZ{B} = \{ (3,2), (4,1), (5,0) \}$,
which has a unique  dec-min element $m_{\ZZ}=(3,2)$.
The defining supermodular function $p$ is given by
\[
 p(\emptyset)= 0,  
\quad
 p(\{ s_{1} \})= 3,  
\quad
 p(\{ s_{2} \})= 0,  
\quad
 p(\{ s_{1},s_{2} \}) = 5.
\]
In Step~1 of the Groenevelt-type decomposition algorithm, we obtain
$a = \allowbreak  \lfloor p(S) / |S| \rfloor \allowbreak 
 = \lfloor 5 / 2 \rfloor = 2$ and $k = p(S) - a |S| = 1$.
We have two choices for  $x$, namely,
$x\sp{(1)} = (2,3)$ and $x\sp{(2)} = (3,2)$.

The first vector $x\sp{(1)} = (2,3)$ does not belong to 
$\odotZ{B}$, and the vector $y$ in Step~3 is given (uniquely) by $y=(3,3)$,
for which we obtain 
$S_{+} = \{ s_{1} \}$ and $S_{-} = \{ s_{2} \}$
in Step~4.  
In Step~5, the restriction $\odotZ{B_{+}}$ consists of a single number 
(one-dimensional vector) 3,
that is,
$\odotZ{B_{+}} = \{ 3 \}$, 
for which $z\sp{*}_{+} = 3$.
In Step~6, the contraction is given by 
$\odotZ{B_{-}} = \{ 2 \}$,
for which $z\sp{*}_{-} = 2$.
In Step~7, we obtain 
$z\sp{*} = (z\sp{*}_{+}, z\sp{*}_{-}) = (3,2)$,
which is the dec-min element of $\odotZ{B}$.

The second vector $x\sp{(2)} = (3,2)$
is already in the given M-convex set $\odotZ{B}$.
Therefore, the algorithm terminates at Step~2
with 
$z\sp{*} = x\sp{(2)} = (3,2)$.
Note that no recursive calls are involved, which means that 
the dec-min element is computed without decomposing the ground-set.
\finbox
\end{example}

The Groenevelt-type decomposition algorithm
may not find the canonical partition.
The resulting decomposition can be coarser or finer than
the canonical partition, which is shown in the following examples.

\begin{example} \rm \label{EXgroePat}
For the problem of Example~\ref{EXgroedim2},
we have
\begin{equation*} 
 p(X) - \beta |X| = 
   \left\{  \begin{array}{ll}
    0   & (X = \emptyset),  \\
    3 - \beta  & (X = \{ s_{1} \}),  \\
     - \beta   & (X = \{ s_{2} \}),  \\
    5 - 2 \beta  & (X = \{ s_{1}, s_{2} \}).  \\
             \end{array}  \right.
\end{equation*}
By Proposition~\ref{PRaltcano}, 
there are two essential values $\beta_{1} = 3$ and $\beta_{2} = 2$, and
the canonical partition is a bipartition
$\{ S_{1}, S_{2} \}$ 
given by
\begin{align*}
S_{1} 
&= L(\beta_{1}-1) - L(\beta_{1})
= L(2) - L(3)
= \{ s_{1} \} - \emptyset
= \{ s_{1} \},
\\ 
S_{2} 
&= L(\beta_{2}-1) - L(\beta_{2})
= L(1) - L(2)
= \{ s_{1}, s_{2} \} - \{ s_{1} \}
= \{ s_{2} \}.
\end{align*}
The first choice $x\sp{(1)} = (2,3)$ in Example~\ref{EXgroedim2}
results in the decomposition with
$S_{+} = \{ s_{1} \}$ and $S_{-} = \{ s_{2} \}$,
which coincides with the canonical partition,
whereas the second choice $x\sp{(2)} = (3,2)$ 
does not decompose $S$ at all, that is,
results in a decomposition coarser than the canonical partition.
\finbox
\end{example}

\begin{example} \rm  \label{EXgroePat2}
Let $S = \{ s_{1} , s_{2} , s_{3} , s_{4} \}$ and 
$\odotZ{B} = \{ (1,0,1,0),  (1,0,0,1), (0,1,0,1), \allowbreak  (0,1,1,0)   \}$,
in which every element is dec-min.
The defining supermodular function $p$ is given by
\begin{align*}
& p(\emptyset)= 0,  
\quad
 p(S)= 2,  
\quad
 p(\{ s_{i} \})= 0,  
\quad
 p(S - \{ s_{i} \}) = 1
\quad (i=1,\ldots,4),
\\ &
 p(\{ s_{1}, s_{2} \}) = p(\{ s_{3}, s_{4} \}) = 1,
\quad   p(X)=0 \ \ \mbox{for other $X$ with $|X|=2$}.
\end{align*}
We can verify, as in Example~\ref{EXgroePat},
that the canonical partition
is a trivial partition consisting of $S$ itself.
The Groenevelt-type decomposition algorithm may result in a finer decomposition.
In Step~1,
we get
$a = \lfloor p(S) / |S| \rfloor = \lfloor 2 / 4 \rfloor = 0$ and $k = p(S) - a |S| = 2$.
Suppose we have chosen 
$x = (0,0,1,1)$, which does not belong to $\odotZ{B}$.
For this $x$, we may take 
$y = (0,1,1,1) \in \odotZ{Q}$
in Step~3, 
for which $S_{+} = \{ s_{1} , s_{2} \}$ and
$S_{-} = \{ s_{3} , s_{4} \}$.
We have 
$\odotZ{B_{+}} = \{ (1,0),  (0,1) \}$
and may take 
$z\sp{*}_{+} = (1,0)$,
whereas we have
$\odotZ{B_{-}} = \{ (1,0),  (0,1) \}$
and may take 
$z\sp{*}_{-} = (1,0)$.
Then we obtain 
$z\sp{*} = (z\sp{*}_{+}, z\sp{*}_{-}) = (1,0,1,0)$.
The resulting decomposition of $S$ is a bipartition 
$\{ \{ s_{1}, s_{2} \}, \{ s_{3}, s_{4} \}  \}$,
which is finer than the canonical partition.
It is noteworthy that
the first member $S_{+} = \{ s_{1} , s_{2} \}$ 
is not even a $z\sp{*}$-top set.
\finbox
\end{example}

\subsection{Fujishige-type decomposition algorithm}
\label{SCdecalgF}

In this section, we develop another decomposition algorithm 
for square-sum minimization on an M-convex set
using the framework of Fujishige \cite[Section~8.2]{Fuj05book}
for separable convex minimization on a base-polyhedron.
The proposed algorithm is based on the natural idea
to apply the framework of \cite[Section~8.2]{Fuj05book} 
to the piecewise-linear extension
$\overline{W}(x)$ in \eqref{pclinW}.
To ensure strong polynomiality, however,
we need to devise a non-trivial gadget to cope with complications arising from integrality.
The relation of the proposed algorithm
to the framework of \cite[Section~8.2]{Fuj05book}
is explained in Remark~\ref{RMdecalgFder} at the end of this section.

By definition,
an element $z$ of an M-convex set $\odotZ{B}$
admits no 1-tightening step if it satisfies the condition:
\begin{equation} \label{decminoptexcZ}
 z(t) \geq z(s) + 2
\ \Longrightarrow \ 
 z + \chi_{s} - \chi_{t} \notin \odotZ{B} .
\end{equation}
This condition, 
Condition (B)
in Theorem~\ref{THdecmincharZ},
is necessary and sufficient for $z \in \odotZ{B}$ to be a dec-min element of $\odotZ{B}$.
It is worth noting that this condition 
coincides with the local optimality condition
\cite[Theorem 6.26]{Mdcasiam} 
for M-convex function minimization applied to 
the function $W(x)$ over $\odotZ{B}$.

We consider a relaxation of the condition \eqref{decminoptexcZ} on $z$,
which, for any given integer $a$, requires that 
\begin{equation}  \label{basedecalglocoptZ2}
 z(s) \leq a,  \  \  z(t) \geq a + 1,  \  \  z(t) \geq z(s) + 2
\ \Longrightarrow \ 
 z + \chi_{s} - \chi_{t} \notin \odotZ{B}.
\end{equation}
Obviously, the condition
\eqref{decminoptexcZ} is stronger than
\eqref{basedecalglocoptZ2} for any fixed $a$.

To consider algorithmic aspects of \eqref{basedecalglocoptZ2}, 
it is convenient to relate \eqref{basedecalglocoptZ2} 
to convex minimization.
Define functions 
\begin{align}
g_{a}(k) &  := \max(a - k, 0,  k - a -1) 
\qquad (k \in \ZZ) ,
  \label{devfnZdef}
\\ 
G_{a}(x) 
& := \sum_{s \in S} g_{a}(x(s)) 
\nonumber \\ &
= \sum_{s \in S} \max(a - x(s), 0, x(s) - a -1)
\qquad (x \in \ZZ\sp{S}).
\label{basedecalgS2fnZ}
\end{align}
Then, \eqref{basedecalglocoptZ2} 
can be recognized as a
local optimality condition
for the minimization of $G_{a}(x)$ over $\odotZ{B}$.

\begin{proposition} \label{PRdecalgFS2minGZ}
$z \in \odotZ{B}$ satisfies \eqref{basedecalglocoptZ2}
if and only if $z$ is a minimizer of $G_{a}$ 
over $\odotZ{B}$.
\end{proposition}
\begin{proof}
Since $G_{a}$ is a separable convex function,
its restriction to the M-convex set $\odotZ{B}$
is an M-convex function (see \cite[(6.31)]{Mdcasiam}).
By the optimality criterion \cite[Theorem 6.26]{Mdcasiam} for M-convex functions,
$z$ is a minimizer of $G_{a}$ over $\odotZ{B}$
if and only if, for any $s,t \in S$, we have the local optimality that
\[
 z + \chi_{s} - \chi_{t} \in \odotZ{B}
\ \Longrightarrow \ 
 G_{a}(z + \chi_{s} - \chi_{t}) \geq G_{a}(z) ,
\]
which is equivalent to
\[
 G_{a}(z + \chi_{s} - \chi_{t}) < G_{a}(z)
\ \Longrightarrow \ 
 z + \chi_{s} - \chi_{t} \notin \odotZ{B}.
\]
Here we have
\begin{align*}
& G_{a}(z + \chi_{s} - \chi_{t}) < G_{a}(z) 
\\ & \iff
  \left(g_{a}(z(s) + 1) - g_{a}(z(s)) \right) + \left(g_{a}(z(t) - 1) - g_{a}(z(t)) \right)   < 0 
\\ & \iff
[ z(s) \leq a, \ z(t) \geq a + 1,  \  z(t) \geq z(s) + 2 ].
\end{align*}
Hence follows the claim of Proposition~\ref{PRdecalgFS2minGZ}.
\qedJIAM
\end{proof}

The following proposition shows that an integral vector $z \in \odotZ{B}$ 
satisfying \eqref{basedecalglocoptZ2}
can be computed in strongly polynomial time.

\begin{proposition} \label{PRdecalgFS2polytimeZ}
For any given $a \in \ZZ$, 
the function $G_{a}$ can be minimized over $\odotZ{B}$ in strongly polynomial time.
Equivalently,
a member of $\odotZ{B}$ satisfying \eqref{basedecalglocoptZ2},
for any given $a \in \ZZ$,
can be found in strongly polynomial time.
\end{proposition}

\begin{proof}
We express a vector $x \in \ZZ\sp{S}$
as
$x = x\sp{1} + x\sp{2} + x\sp{3}$,
where $x\sp{1}$, $x\sp{2}$, and $x\sp{3}$ are integer vectors and
\[
 x\sp{1}(s) \leq a, 
\quad
 0 \leq x\sp{2}(s) \leq 1, 
\quad
 0 \leq x\sp{3}(s) 
\qquad (s \in S).
\]
Consider three disjoint copies
$S\sp{1}$, $S\sp{2}$, $S\sp{3}$ of the ground-set $S$,
and regard
$x\sp{i}$
as a vector on $S\sp{i}$ 
for $i=1,2,3$.
The set defined by
\begin{align*}
\odotZ{B'} = &
 \{ (x\sp{1}, x\sp{2},x\sp{3}) \in \ZZ\sp{S\sp{1} \cup S\sp{2} \cup S\sp{3} } 
 \mid 
\\ & \quad
 x\sp{1} + x\sp{2} + x\sp{3} \in \odotZ{B}, \ 
 x\sp{1} \leq a {\bm 1}, \  {\bm 0} \leq x\sp{2} \leq {\bm 1}, \ {\bm 0} \leq x\sp{3}  \}
\end{align*}
is an M-convex set 
(which can be proved directly or by using \cite[Proposition~3.3]{Mopernet21}).
Furthermore, the minimization of
$G_{a}(x)$ over $\odotZ{B}$
is reduced to the minimization of 
$\sum_{s \in S_{1}}(a - x\sp{1}(s) ) + \sum_{s \in S_{3}}x\sp{3}(s)$
over $\odotZ{B'}$.
The latter problem is a linear optimization over a base-polyhedron,
and the required subroutines for $\odotZ{B'}$
can be realized from those of the given M-convex set $\odotZ{B}$
with the aid of a submodular function minimization algorithm. 
 For technical details, the reader is referred to
  \cite{Fra11book}
(Section~14.3 (in particular, Theorem~14.3.39) and Section~14.5).
\qedJIAM
\end{proof}

\medskip

We are now ready to present the Fujishige-type decomposition algorithm 
for computing a dec-min element of $\odotZ{B}$.
We assume that $B$ is given
by an integer-valued supermodular function $p$
as $B=B'(p)$ in \eqref{basepolysupermod}.

\medskip

\noindent
{\bf Fujishige-type decomposition algorithm for square-sum on an M-convex set $\odotZ{B}$}%
\\
1: Set  $a := \lfloor p(S) / |S| \rfloor$.
\\
2: Find $z \in \odotZ{B}$ that satisfies \eqref{basedecalglocoptZ2}
  (by the method in the proof of Proposition~\ref{PRdecalgFS2polytimeZ}).
\\
3: Define $S_{+}$, $S_{-}$, $S_{0}$ by
\begin{align}
 S_{+} &:= \bigcup_{t \in S: z(t) \geq a+2} 
 \{ s \in S \mid 
 z + \chi_{s} - \chi_{t} \in \odotZ{B}   \},
 \label{basedecalgS+Z}
\\
 S_{-} &:= \bigcup_{s \in S: z(s) \leq a-1} 
 \{ t \in S \mid 
 z + \chi_{s} - \chi_{t} \in \odotZ{B}  \},
 \label{basedecalgS-Z}
\\
 S_{0} &:= S - ( S_{+} \cup S_{-} ),
 \label{basedecalgS0Z}
\end{align}
\phantom{3:}
and let $z\sp{*}(s) := z(s)$ for each $s \in S_{0}$.  
\\
4: 
   If $S_{+} \not= \emptyset$, then apply the present algorithm recursively to
  the restriction to $S_{+}$.  
\\ \phantom{4:}
 If $S_{-} \not= \emptyset$, then apply the present algorithm recursively to
  the contraction to $S_{-}$.  
\\
5: 
If both  $S_{+}$ and $S_{-}$ are empty, then stop.
\finbox 

\medskip

When the algorithm terminates in Step~5, the component values
$z\sp{*}(s)$ have already been defined for all $s \in S$. 
The computed vector $z\sp{*}$ 
is a dec-min element of $\odotZ{B}$,
which we show in Proposition~\ref{PRdecalgFoutput}.
Step~3 of the algorithm
defines the component values $z\sp{*}(s)$ for $s \in S_{0}$
(the subset $S_{0}$ may be an empty),
whereas the component values on $S_{+}$ and $S_{-}$ are determined
within the recursive calls in Step~4.
The recursive call to the restriction to $S_{+}$
means applying the above algorithm
to $B_{+} = B'(p_{+})$ on $S_{+}$,
where 
$p_{+}$ is the supermodular function on $S_{+}$ defined by
$p_{+}(X) = p(X)$ for $X  \subseteq S_{+}$.
The recursive call to the contraction to $S_{-}$
means applying the above algorithm 
to $B_{-} = B'(p_{-})$ on $S_{-}$,
where 
$p_{-}$ is the supermodular function  on $S_{-}$ defined by
$p_{-}(X) = p(X \cup (S - S_{-})) - p(S - S_{-})$ for $X  \subseteq S_{-}$.

The following proposition implies 
the termination of the algorithm with at most $|S|$ recursive calls.

\begin{proposition} \label{PRdecalgFS-S+Z}
$S_{+} \not= S$ and $S_{-} \not= S$.
\end{proposition}
\begin{proof}
(This is an adaptation of the argument of 
\cite[Section~8.2]{Fuj05book} to the discrete case.) \
To show $S_{+} \not= S$ by contradiction, assume $S_{+} = S$.
By (\ref{basedecalgS+Z}), this implies that 
for every $s \in  S$, there exists $t \in S$
such that  $z(t) \geq a+2$ and
$z + \chi_{s} - \chi_{t} \in \odotZ{B}$. 
Combining  this with (\ref{basedecalglocoptZ2}) 
we have
$z(s) \geq a+1$
for all $s \in S$,
which implies
$p(S)=\sum [ z(s):  s \in S ] \geq (a+1) |S|$. 
We cannot have equality here,
because, by (\ref{basedecalgS+Z}),
our assumption $S_{+} = S$  implies
$z(t_{0}) \geq a+2$ for some $t_{0} \in S$.
Therefore,
$p(S) > (a+1) |S|$, 
which is a contradiction to
$a = \lfloor p(S) / |S| \rfloor$.

To show $S_{-} \not= S$ by contradiction, assume $S_{-} = S$.
By (\ref{basedecalgS-Z}), this implies that 
for every $t \in  S$, there exists $s \in S$
such that  $z(s) \leq a-1$ and
$z + \chi_{s} - \chi_{t} \in \odotZ{B}$. 
Combining  this with (\ref{basedecalglocoptZ2}) 
we have $z(t) \leq a$ for all $t \in S$,
which implies
$p(S)=\sum [ z(t) : t \in S ] \leq a |S|$. 
We cannot have equality here,
because, by (\ref{basedecalgS-Z}),
our assumption $S_{-} = S$ implies $z(s_{0}) \leq a-1$
for some $s_{0} \in S$.
Therefore, $p(S) < a |S|$, which is a contradiction to 
$a = \lfloor p(S) / |S| \rfloor$.
\qedJIAM
\end{proof}

The correctness of the algorithm is established in the following proposition.

\begin{proposition} \label{PRdecalgFoutput}
The output $z\sp{*}$ of the Fujishige-type decomposition algorithm
is a minimizer of square-sum $W(x)$ over $\odotZ{B}$,
or equivalently, $z\sp{*}$ is a dec-min element of $\odotZ{B}$.
\end{proposition}
\begin{proof}
(This is an adaptation of the argument in \cite[Section~8.2]{Fuj05book} to the discrete case.)
We prove that the output of the algorithm is a square-sum minimizer
by induction on the size $|S|$ of the ground-set $S$.
We have $|S_{+}| < |S|$  and $|S_{-}| < |S|$ by Proposition~\ref{PRdecalgFS-S+Z}.
By the definition in Step~3 and the property  \eqref{basedecalglocoptZ2} of $z$, 
we have
\begin{align}
&
\{ s : z(s) \geq a +2 \}
\ \subseteq S_{+} \subseteq \
\{ s : z(s) \geq a +1 \},
\label{decalgS+val}
\\ &
\{ s : z(s) \leq a-1 \} 
\ \subseteq S_{-} \subseteq \
\{ s : z(s) \leq a \},
\label{decalgS-val}
\end{align}
which imply that
$S_{+}$ and $S_{-}$ are disjoint and 
\begin{equation}
S_{0} \subseteq \{ s : a \leq z(s) \leq a+1  \} .
\label{decalgS0val}
\end{equation}

First we show that the output $z\sp{*}$ of the algorithm satisfies
\begin{equation}
 z\sp{*}(s)  \ \begin{cases}
 \geq a+1 & (s \in S_{+} ) ,
\\
 \leq a & (s \in S_{-} ) ,
\\
\in \{ a, a+1 \} & (s \in S_{0} ). 
\end{cases}
\label{decalgz*val}
\end{equation}
To prove 
$z\sp{*}(s) \geq a + 1$ for $s \in S_{+}$,
let $z_{+}\sp{*}$ denote the subvector of $z\sp{*}$ on $S_{+}$,
which is the outcome of the recursive call to 
the restriction to $S_{+}$.
By the induction hypothesis,
$z_{+}\sp{*}$ minimizes 
the square-sum
$\sum [ x(s)\sp{2}: s \in S_{+} ]$
over the M-convex set $\odotZ{B_{+}}$,
where $B_{+} = B'(p_{+})$ is the restriction to $S_{+}$.
Then 
Theorems \ref{THdecmincharZ} and \ref{THdecminZsqsum}
show that 
$z_{+}\sp{*}$ is an inc-max element of $\odotZ{B_{+}}$,
whereas the subvector of $z$ on $S_{+}$ is a member of $\odotZ{B_{+}}$
satisfying 
$z(s) \geq a + 1$ for all $s \in S_{+}$.
Therefore, 
$z\sp{*}(s) \geq a + 1$ for $s \in S_{+}$.
The second case, $z\sp{*}(s) \leq a$ for $s \in S_{-}$,
can be proved symmetrically. Namely,
by the induction hypothesis,
the subvector $z_{-}\sp{*}$ of $z\sp{*}$ on $S_{-}$
minimizes the square-sum
over the M-convex set $\odotZ{B_{-}}$
defined by the contraction $B_{-} = B'(p_{-})$ to $S_{-}$,
and hence $z_{-}\sp{*}$ is a dec-min element of $\odotZ{B_{-}}$
by Theorem~\ref{THdecminZsqsum}.
This implies
$z\sp{*}(s) \leq a$ for $s \in S_{-}$,
since the subvector of $z$ on $S_{-}$ is a member of $\odotZ{B_{-}}$
satisfying
$z(s) \leq a$ for all $s \in S_{-}$.
The third case,
$z\sp{*}(s) \in \{ a, a+1 \}$ for $s \in S_{0}$,
is obvious from \eqref{decalgS0val} and the definition
$z\sp{*}(s) = z(s)$ for $s \in S_{0}$ in Step~3.

We will show that $z\sp{*}$ satisfies the condition \eqref{decminoptexcZ}
for all $(s,t)$ with $s, t \in S_{+} \cup S_{0} \cup S_{-}$.
Suppose that $z\sp{*}(t) \geq z\sp{*}(s) + 2$.
By \eqref{decalgz*val} we may assume
$(s,t) \notin (S_{+} \cup S_{0}) \times (S_{-} \cup S_{0})$.

The condition \eqref{decminoptexcZ} holds 
when $(s,t) \in S_{+} \times S_{+}$,
since the subvector $z_{+}\sp{*}$ 
is an inc-max (and hence dec-min) element of $\odotZ{B_{+}}$,
as already mentioned.
Similarly, \eqref{decminoptexcZ} holds 
when $(s,t) \in S_{-} \times S_{-}$,
since the subvector $z_{-}\sp{*}$ 
is a dec-min element of $\odotZ{B_{-}}$.

It remains to consider the other three cases:
$(s,t) \in S_{-} \times S_{+}$,
$(s,t) \in S_{0} \times S_{+}$, and
$(s,t) \in S_{-} \times S_{0}$.
Since $z_{+}\sp{*}$ is a base of the restriction $B_{+}$,
we have
$\widetilde {z\sp{*}}(S_{+}) = p(S_{+})$, that is,
$S_{+}$ is $z\sp{*}$-tight with respect to $p$.
It then follows that 
$z\sp{*} + \chi_{s} - \chi_{t} \notin \odotZ{B}$
if $s \in S - S_{+}$ and $t \in S_{+}$.
Hence, the condition \eqref{decminoptexcZ} holds for 
$(s,t) \in (S-S_{+}) \times S_{+} = (S_{-} \cup S_{0}) \times S_{+}$.
Similarly, since
$z_{-}\sp{*}$ is a base of the contraction $B_{-}$,
we have
$\widetilde {z\sp{*}}(S_{-}) = p_{-}(S_{-}) = p(S) - p(S - S_{-}) 
= \widetilde {z\sp{*}}(S) - p(S - S_{-})$, 
which shows that
$S - S_{-}$ is $z\sp{*}$-tight with respect to $p$.
Hence
$z\sp{*} + \chi_{s} - \chi_{t} \notin \odotZ{B}$
if $s \in S_{-}$ and $t \in  S- S_{-}$. 
Hence, the condition \eqref{decminoptexcZ} holds for 
$(s,t) \in S_{-} \times (S-S_{-}) = S_{-} \times  (S_{+} \cup S_{0})$.
This completes the proof of Proposition~\ref{PRdecalgFoutput}.
\qedJIAM
\end{proof}

The Fujishige-type decomposition algorithm can be executed in strongly polynomial time.
The key fact here is that we can carry out Step~2, 
which is characteristic of the discrete case,
in polynomial time.

\begin{proposition} \label{PRdecalgFpolytime}
The Fujishige-type decomposition algorithm
computes a 
\linebreak
square-sum minimizer (i.e., dec-min element) of an M-convex set
in strongly polynomial time.
\end{proposition}
\begin{proof}
By Proposition~\ref{PRdecalgFS-S+Z},
the number of recursive calls is bounded by $|S|$.
In each call of the algorithm, the vector $z$ in 
Step~2 can be found in strongly polynomial time by 
Proposition~\ref{PRdecalgFS2polytimeZ}.
The subsets $S_{+}$ and $S_{-}$ 
in Step~3 can be determined in strongly polynomial time
by a standard method using submodular function minimization
\cite{Fuj05book,Sch03}.
\qedJIAM
\end{proof}

\begin{remark} \rm  \label{RMdecalgFsplit}
The subsets $S_{+}$ and $S_{-}$ constructed 
in the Fujishige-type decomposition algorithm
have crucial properties that each of  
$S_{+}$ and $S - S_{-}$ is a $z\sp{*}$-top and 
$z\sp{*}$-tight set (with respect to $p$),
where the former property is obvious from \eqref{decalgz*val} and
the latter property is shown in the proof of Proposition~\ref{PRdecalgFoutput}.
In addition, 
$z\sp{*}$ is near-uniform on their difference $S_{0} = (S - S_{-}) -  S_{+}$,
where $S_{0}$ may possibly be empty.
It follows from these properties that 
an integral vector $z = (z_{+}, z_{0}, z_{-})$ 
is a dec-min element of $\odotZ{B}$ 
if and only if 
$z_{+}$ is a dec-min element of $\odotZ{B_{+}}$,
$z_{-}$ is a dec-min element of $\odotZ{B_{-}}$,
and
$z_{0}$ is a near-uniform element of $\odotZ{B_{0}}$,
where $B_{0} = B'(p_{0})$ 
with
$p_{0}(X) = p(X \cup S_{+}) - p(S_{+})$ for $X  \subseteq S_{0}$.
This fact justifies the recursive calls in Step~4.
\finbox
\end{remark}

\begin{remark} \rm  \label{RMdecalgFchain}
Condition (C) in Theorem~\ref{THdecmincharZ}
refers to a chain 
$C_{1}\subset C_{2}\subset \cdots \subset C_{\ell}$ 
to characterize a dec-min element.
The subsets $S_{+}$ and $S_{-}$ constructed in the algorithm
correspond to this chain as follows.
First note that
$S_{+} \subseteq S - S_{-}$, where equality may occur.
As mentioned in Remark~\ref{RMdecalgFsplit},
both $S_{+}$ and $S - S_{-}$ are 
$z\sp{*}$-top and $z\sp{*}$-tight sets,
and moreover, $z\sp{*}$ is near-uniform on their difference $S_{0} = (S - S_{-}) -  S_{+}$.
Through the recursive calls to $S_{+}$ and $S_{-}$,
the algorithm constructs, in effect, the chain 
$C_{1}\subset C_{2}\subset \cdots \subset C_{\ell}$.
More precisely,
if $S_{+}$ and $S - S_{-}$ are distinct,
they are consecutive members 
of the chain; otherwise,
$S_{+}$ $(=S - S_{-})$ is a member of the chain.
\finbox
\end{remark}

\begin{remark} \rm  \label{RMdecalgFpat}
The Fujishige-type decomposition algorithm
may not find the canonical chain,
in spite of the fact explained in Remark~\ref{RMdecalgFchain}.
In Example~\ref{EXgroedim2}, for instance,
we have
$S = \{ s_{1}, s_{2} \}$ and
$a = \lfloor p(S) / |S| \rfloor = \lfloor 5 / 2 \rfloor = 2$. 
The element of $\odotZ{B}$ satisfying \eqref{basedecalglocoptZ2}
in Step~2 is given (uniquely) by $z=(3,2)$,
for which
$S_{+} = \emptyset$, $S_{-} = \emptyset$, and $S_{0} = S$
in Step~3.
Thus the ground-set $S$ is not decomposed at all, 
which corresponds to a trivial chain consisting of a single member $S$.
In contrast,
as we have seen in Example~\ref{EXgroePat},
the canonical partition is a bipartition 
$\{ S_{1}, S_{2} \}$ with $S_{1} = \{ s_{1} \}$ and $S_{2} = \{ s_{2} \}$,
which corresponds to the canonical chain:
$\{ s_{1} \} \subset  \{ s_{1}, s_{2} \}$.
\finbox
\end{remark}

\begin{remark} \rm  \label{RMdecalgFder}
We explain here how our Fujishige-type decomposition algorithm is
derived from the framework in \cite[Section~8.2]{Fuj05book}
with additional integrality considerations.
We apply the framework to the piecewise-linear function
$\overline{W}(x) = \sum [ \overline{\varphi} ( x(s) ) : s \in S ]$
defined in \eqref{pclinW},
where $\overline{\varphi}: \RR \to \RR$ 
is the piecewise-linear extension \eqref{pclinphi} of 
$\varphi (k)=k\sp{2}$  $(k \in \ZZ)$.

The notation of \cite[Section~8.2]{Fuj05book} is as follows.
The ground-set is $E$, and a base-polyhedron $B$ is described by a
submodular function $f$.
For $x \in B$, the smallest $x$-tight set (with respect to $f$) containing $e \in E$
is denoted as ${\rm dep}(x,e)$. 
The objective function to be minimized on $B$ is a separable convex function
$\sum_{e \in E} w_{e}(x(e))$,
where each $w_{e}: \RR \to \RR$ is a real-valued convex function on $\RR$.
The left and right derivatives of $w_{e}$ at $\xi \in \RR$ are denoted, respectively, by
${w_{e}}\sp{-}(\xi)$ and ${w_{e}}\sp{+}(\xi)$.
For any $\eta \in \RR$, 
$[ {i_{e}}\sp{-}(\eta) , {i_{e}}\sp{+}(\eta) ]$
denotes the set (interval) of  minimizers $\xi$ of function
$w_{e}(\xi) - \eta \xi$.
In our problem to minimize $\overline{W}$, we have
\begin{align}  
& w_{e}(\xi) = \overline{\varphi}(\xi), 
\quad
{i_{e}}\sp{-}(\eta) = \left\lceil \frac{\eta -1 }{2} \right\rceil ,
\quad
 {i_{e}}\sp{+}(\eta) = \left\lfloor \frac{\eta + 1}{2} \right\rfloor 
\label{basedecalgssubdiffZ2E}
\end{align}
for all $e \in E$.

Step 1 in \cite[page 258]{Fuj05book} says: 
\[
\mbox{Choose $\eta \in \RR$ such that } \ 
\sum_{e \in E} {i_{e}}\sp{-}(\eta) \leq f(E) \leq \sum_{e \in E} {i_{e}}\sp{+}(\eta).
\]
It follows from  (\ref{basedecalgssubdiffZ2E})
that the choice of $\eta = 2a +1 $
with $a = \lfloor f(E) / |E| \rfloor$
satisfies this condition;
then ${i_{e}}\sp{-}(\eta) = a$ and ${i_{e}}\sp{+}(\eta) = a + 1$.
This explains our Step~1 
to set $a := \lfloor p(S) / |S| \rfloor$.

Step 2 in \cite[page 258]{Fuj05book} is a tricky step where
the substantial condition \eqref{basedecalglocoptZ2} is hidden.
This step requires to find a base $x \in B$ such that, for each $s, t \in E$,
\begin{enumerate}
\item
 if $w_{s}^{+}(x(s)) < \eta$ and $w_{t}^{-}(x(t)) > \eta$,
then we have $t \not\in {\rm dep}(x,s)$,

\item
 if $w_{s}^{+}(x(s)) < \eta$ and $w_{t}^{-}(x(t)) = \eta$,
and $t \in {\rm dep}(x,s)$, 
then for any $\alpha > 0$ we have 
$w_{t}^{-}(x(t) - \alpha) < \eta$, i.e.,
$x(t) = {i_{t}}\sp{-}(\eta)$,

\item
if $w_{s}^{+}(x(s))= \eta$ and $w_{t}^{-}(x(t)) > \eta$,
and $t \in {\rm dep}(x,s)$, 
then for any $\alpha > 0$ we have 
$w_{s}^{+}(x(s) + \alpha) > \eta$, i.e.,
$x(s) = {i_{s}}\sp{+}(\eta)$.
\end{enumerate}
In our case,
$B$ is an integral base-polyhedron,
$x \in \odotZ{B}$,
and
$w_{s}(\xi) = \overline{\varphi}(\xi)$, for which
$w_{s}^{+}(\xi) = 2 \xi +1$ and $w_{s}^{-}(\xi) = 2 \xi -1$
for integer $\xi$.
Under the integrality requirement, the first condition is replaced by
\begin{equation} \label{decalgFopt1}
 x(s) < a,   \  \  x(t) > a + 1 
\ \Longrightarrow \ 
 x + \chi_{s} - \chi_{t} \not\in \odotZ{B}.
\end{equation}
A literal translation of the second condition results in the condition
\[
 x(s) < a,   \  \  x(t) = a + 1 , \  \    x + \chi_{s} - \chi_{t} \in \odotZ{B} 
\ \Longrightarrow \ 
x(t) = a, 
\]
which should be interpreted as
\begin{equation} \label{decalgFopt2}
 x(s) < a,   \  \  x(t) = a + 1 
\ \Longrightarrow \ 
 x + \chi_{s} - \chi_{t} \not\in \odotZ{B} .
\end{equation}
Similarly, the third condition is replaced by
\begin{equation} \label{decalgFopt3}
 x(s) = a,   \  \  x(t) > a + 1 
\ \Longrightarrow \ 
 x + \chi_{s} - \chi_{t} \not\in \odotZ{B}.
\end{equation}
The combination of \eqref{decalgFopt1}--\eqref{decalgFopt3} 
is equivalent to
\begin{equation*} 
 x(s) \leq a,  \  \  x(t) \geq a + 1,  \  \  x(t) \geq x(s) + 2
\ \Longrightarrow \ 
 x + \chi_{s} - \chi_{t} \not\in \odotZ{B} ,
\end{equation*}
which coincides with the condition (\ref{basedecalglocoptZ2}) in our Step~2.

The remaining steps of our algorithm is a straightforward 
translation of the corresponding steps in \cite[page 258]{Fuj05book}
with obvious integrality requirements.
\finbox
\end{remark}

\paragraph{Acknowledgement} \ 
We thank Satoru Iwata and Akiyoshi Shioura for discussion about algorithms,
and Arie Tamir for indicating references.
We also thank the anonymous referee for helpful comments.
Insightful questions posed by Tam{\'a}s Kir{\'a}ly led to 
Theorem~\ref{THmnormconvcombdmS}.
The research was partially supported by the
National Research, Development and Innovation Fund of Hungary
(FK\_18) -- No. NKFI-128673,
and by JSPS KAKENHI Grant Number JP20K11697.




\end{document}